\documentclass[a4paper, 11pt, oneside, onecolumn,sort&compress]{elsarticle}
\usepackage{amssymb}
\usepackage{amsfonts}
\usepackage{amscd}
\usepackage{mathrsfs}
\usepackage[reqno]{amsmath}
\usepackage{amsthm,upref,amscd}
\usepackage{amssymb,amsmath}
\usepackage[usenames,dvipsnames]{color}
\usepackage[colorlinks=true,linkcolor=Red,citecolor=Green]{hyperref}
\usepackage{bm}

\usepackage[usenames,dvipsnames]{color}
\usepackage[colorlinks=true,linkcolor=Red,citecolor=Green]{hyperref}
\usepackage{bm,enumitem}

\newtheorem{theorem}{Theorem}[section]
\newtheorem{lemma}{Lemma}[section]

\newdefinition{example}{Example}[section]
\newdefinition{remark}{Remark}[section]
\newdefinition{definition}{Definition}[section]

\def\cM{{\mathcal M}}
\def\R{\mathbb{R}}

\def\N{\mathbb{N}}
\def\va{\varepsilon}

\numberwithin{equation}{section} \journal{}

\begin{document}
\begin{frontmatter}

\title{\textbf{Normalized solutions for some
quasilinear elliptic equation
with critical Sobolev exponent}\tnoteref{label1}} \tnotetext[label1]{Projects supported by National Natural Science Foundation of China (Grant Nos. 12271313, 12071266, 12026217, 12026218, 12101376), Fundamental Research Program of Shanxi Province (202203021211300, 202203021211309, 20210302124528) and Shanxi Scholarship Council of China (Grant No. 2020-005).}
\author{Xiaojing Feng} \ead{fengxj@sxu.edu.cn}
\author{Yuhua Li\corref{cor1}} \ead{yhli@sxu.edu.cn}
\cortext[cor1]{Corresponding author.}
\address{School of Mathematical Sciences, Shanxi University, Taiyuan 030006, Shanxi, P.R. China}
\begin{abstract}
Consider the equation
\begin{equation*}
-\Delta_p u
=\lambda |u|^{p-2}u+\mu|u|^{q-2}u+|u|^{p^\ast-2}u\ \ {\rm in}\ \R^N
\end{equation*}
under the normalized constraint
$$\int_{ \R^N}|u|^p=c^p,$$
where $-\Delta_pu={\rm div} (|\nabla u|^{p-2}\nabla u)$, $1<p<N$, $p<q<p^\ast=\frac{Np}{N-p}$, $c,\mu>0$ and $\lambda\in\R$.
In the purely $L^p$-subcritical case,
we obtain the existence of ground state solution by virtue of truncation technique,
and obtain multiplicity of normalized
solutions. In the purely $L^p$-critical and supercritical
case, we drive the existence of positive ground state solution, respectively.
 Finally, we investigate the asymptotic
behavior of ground state solutions obtained above as $\mu\to0^+$.\\[1mm]
\end{abstract}
\end{frontmatter}

\section{Introduction and main results}

The aim of this paper is to study the normalized solutions of some $p$-Laplace equation with combined power nonlinearities
\begin{equation}\label{eq1.1}
-\Delta_p u
=\lambda |u|^{p-2}u+\mu|u|^{q-2}u+|u|^{p^\ast-2}u\ \ {\rm in}\ \R^N,
\end{equation}
where $-\Delta_pu={\rm div} (|\nabla u|^{p-2}\nabla u)$, $1<p<N$, $p<q<p^\ast=\frac{Np}{N-p}$, $\lambda,\,\mu\in\R$.
Before
we make precise statements, let us comment on some works which motivated this one.
It is well known \eqref{eq1.1} is a special form of the equation
\begin{equation}\label{eq1.2}
-\Delta_p u
=\lambda |u|^{p-2}u+f(u)\ \ {\rm in}\ \R^N,
\end{equation}
where $f(u)$  is a general nonlinearity.  In general,
problem \eqref{eq1.2} can be seen as the stationary counterpart of evolution equations with nonlinear
diffusion. The $p$-Laplace equation arises in a variety of physical phenomena, for instance,  in the study of no-Newtonian fluids, and in the study of nonlinear elasticity problems, please see \cite{zgw,pl11}
for more details of physical background.
When looking for solutions to \eqref{eq1.2}, a possible choice is to fix $\lambda\in \R$ and to search for solutions to \eqref{eq1.2} as critical points of the
corresponding energy functional by using variational method,
see for example \cite{dpr,shz,ll11, ly1,coa} in unbounded domains and \cite{gapa1, blw, tbl,jzh}
in bounded domains.

Alternatively, from a physical point of view, it is interesting to find solutions to \eqref{eq1.2} having prescribed mass
\begin{equation}\label{eq1.3}
\int_{ \R^N}|u|^p=c^p\ \text{with}\ c>0.
\end{equation}
In this direction, the parameter $\lambda\in \R$ arises as a Lagrange multiplier, which
depends on the solution and is not a priori given.
The aim of this paper is to establish the existence asymptotic properties of weak solutions of
\eqref{eq1.1} and \eqref{eq1.3}. Here and after, by a solution
we always mean a couple $(u, \lambda)$ which satisfies
\eqref{eq1.1} and \eqref{eq1.3}. One refers to this type of solutions
as to {\it normalized solutions}, since \eqref{eq1.3} imposes
a normalization on the $L^p$-norm of $u$.

In the case of $p=2$, \eqref{eq1.2} can be rewritten as the
semilinear elliptic equation
\begin{equation}\label{eq1.4}
-\Delta u=\lambda u+f(u)\ \ {\rm in}\ \R^N.
\end{equation}
Normalized solutions of \eqref{eq1.4} have achieved considerable attention(see, for examples,
\cite{lj,bs3,bdv,cas1,cas2,lz-12-24,cl1,ms1,ht1}).
Recently, in \cite{ns2} Soave studied \eqref{eq1.4} with
combined nonlinearities, i.e.,
$$f(u)=\mu|u|^{q-2}u+|u|^{p-2}u,$$
where $q<p$ satisfies $2<q\leq\bar{q}=2+4/N\leq p<2^\ast=\frac{2N}{N-2}$. Existence
and asymptotic properties of normalized ground state solutions,
as well as stability/instability results, were established.
Later, in \cite{lz2} Luo and Zhang obtained existence
and nonexistence results of normalized solutions for fractional
Schr\"{o}dinger equations with combined nonlinearities. In particular, Soave
\cite{ns1} investigated the Brezis-Nirenberg problem
\begin{equation}\label{eq1.5}
-\Delta u=\lambda u+\mu|u|^{q-2}u+|u|^{2^\ast-2}u\ \ {\rm in}\ \R^N,
\end{equation}
where $2<q<2^\ast$. The cases of $L^2$-subcritical, $L^2$-critical and $L^2$-supercritical
perturbation were considered respectively. The author proved the
existence and asymptotic properties of normalized ground state
solutions for \eqref{eq1.5}. With the help of positive lower order perturbation
term $\mu|u|^{q-2}u$, the associated energy level can be pulled
down below the noncompactness level and henceforth the existence
result is obtained. Compared with the case of  $p =2$,  there are few papers considered the normalized
solutions of \eqref{eq1.2} with $p\neq 2$ except \cite{gzz, wlzl}.

Motivated by the results mentioned above, we shall establish the
existence and asymptotic properties of normalized
solutions of \eqref{eq1.1} in the current
paper. Let $E=W_r^{1,p}(\R^N)$ and denote by $|\cdot|_s$ the usual
norm of $L^s(\R^N)$ with $s\in[1,\infty)$.
Moreover, we denote by $J_{\mu}: E\to\R$ the energy functional related to \eqref{eq1.1}, given by
$$J_{\mu}(u)=\frac{1}{p}\int_{\R^N}|\nabla u|^p
-\frac{1}{p^\ast}\int_{\R^N}|u|^{p^\ast}-\frac{\mu}{q}\int_{\R^N}|u|^q.$$
It is standard that $J_{\mu}$ is of class $C^1$ in $E$, and any critical points of $J_{\mu}$ constrained to
$$S_c=\{u\in E: |u|_p=c\}$$
give rise to normalized solutions of \eqref{eq1.1} with $\lambda$ as
a Lagrange multiplier. We will focus on the
existence and asymptotic properties of ground state solutions,
the definition of which is as follows.

\begin{definition}
$u\in S_c$ is said to be a ground state solution of
\eqref{eq1.1} if
$$(J_{\mu}|_{S_c})'(u)=0\ \ {\rm and}\ \
J_{\mu}(u)=\inf\left\{J_{\mu}(v):
v\in S_c\ \text{with}\ (J_{\mu}|_{S_c})'(v)=0\right\}.$$
\end{definition}
Next, we recall the Gagliardo-Nirenberg inequality \cite{cmo,mag} as follows.
\begin{lemma}\label{lem1.1}
Let $p<q<p^\ast$, then the following inequality holds
$$\int_{\R^N}|u|^q\leq
C(q)|u|_p^{(1-\gamma_q)q}|\nabla u|_p^{\gamma_q q},
\ \ \text{for all}\ u\in W^{1,p}(\R^N),$$
where $\gamma_q=\frac{N(q-p)}{pq}$ and $C(q)>0$ is the best possible constant in this inequality.
\end{lemma}
In order to obtain the ground state solution, as in \cite{ns1}, we introduce the Poho\v zaev manifold
$$\cM_{c,\mu}=\{u\in S_c:\ P_{\mu}(u)=0\},$$
where
$$P_{\mu}(u)=\int_{\R^N}|\nabla u|^p
-\int_{\R^N}|u|^{p^\ast}-\mu \gamma_q \int_{\R^N}|u|^q.$$
It is clear that $J_{\mu}$ is bounded from below on $\cM_{c,\mu}$
and
$$m(c,\mu)=\inf_{u\in\mathcal {M}_{c,\mu}}J_{\mu}(u)$$
is well defined provided that $\cM_{c,\mu}\neq\emptyset$.
Using Poho\v zaev identity (see \cite{ll11}), the critical points of the functional
$J_{\mu}|_{S_c}$ lie in $\cM_{c,\mu}$.
Hence if $\inf_{\mathcal {M}_{c,\mu}}J_{\mu}$ is achieved by the normalized solution for \eqref{eq1.1}, then it is the a ground state solution of \eqref{eq1.1} on $S(c)$.
In what follows, we will state our main results, for this purpose we will give some notations. Set
$$C'=\left(\frac{p^\ast(p-\gamma_q q)S^{\frac{p^\ast}{p}}}
{p(p^\ast-\gamma_q q)}\right)^{\frac{p-\gamma_q q}{p^\ast-p}}\frac{(p^\ast-p)q}{C(q)p(p^\ast-\gamma_q q)},$$
where $S$ is the optimal constant of the Sobolev embedding
$D^{1,p}(\R^N)\hookrightarrow L^{p^\ast}(\R^N)$. Define
$$\alpha (q)=\left\{\begin{array}{ll}
C', &\text{if}\ p<q<p+p^2/N,\\
\frac{q}{pC({q})},&\text{if}\ q= p+p^2/N,\\
+\infty,&\text{if}\ p+p^2/N< q<p^\ast.
\end{array}\right.$$

\begin{theorem}\label{the1.1}
Let $\mu,\,c>0$ be such that
\begin{equation}\label{eq1.6}
\mu c^{(1-\gamma_q)q}<\alpha(q).
\end{equation}
Assume that $p<q<p+p^2/N$. We obtain \\
$(i)$ \eqref{eq1.1} has a positive
ground state solution $u\in S_c$ with $J_{\mu}(u)<0$ and $u$
is an interior local minimizer of $J_{\mu}$ on the set
$A_{R_0}(c)=\{u\in S_c:|\nabla u|_2<R_0\}$ for a suitable $R_0$.\\
$(ii)$  \eqref{eq1.1} possesses an unbounded sequence of solutions $(u_k,\lambda_k)$
with $|u_k|_p=c$, $\lambda_k<0$ and $J(u_k)\to0^-$ as $k\to\infty$.\\
$(iii)$  if we make additional assumptions $N>2$ and $N^{2/3}<p<3$, then \eqref{eq1.1} also has a positive
normalized solution of mountain pass type.

\end{theorem}

\begin{theorem}\label{the1.2}
Let $\mu,\,c>0$ be such that \eqref{eq1.6} holds. Assume that $p<N^{2/3}$ and $q=p+p^2/N$. Then \eqref{eq1.1} has a positive
ground state solution in $S_c$ for some $\lambda<0$.
\end{theorem}

\begin{theorem}\label{the1.3}Assume that $p<N^{2/3}$ and $p+p^2/N<q<p^\ast$. Then for any $c,\mu>0$, the equation \eqref{eq1.1} has a
positive ground state solution in $S_c$ for some $\lambda<0$.
\end{theorem}

\begin{theorem}\label{the1.4}
Let $u_{c,\mu}\in S_c$ be
the positive ground state solution of \eqref{eq1.1} obtained in Theorems \ref{the1.1}, \ref{the1.2} and \ref{the1.3} with energy
level $m(c,\mu)$.
\begin{itemize}[leftmargin=9mm]
\item[$(i)$] If $p<q<p+p^2/N$, then
$$m(c,\mu)\to 0\ \ \text{and}\ \ |\nabla u_{c,\mu}|_p^p\to 0
\ \ \text{as}\ \mu\to 0^+.$$

\item[$(ii)$] If $p<N^{2/3}$ and $p+p^2/N\leq q<p^\ast$, then
$u_{c,\mu}\rightharpoonup 0$ in $E$,
$$m(c,\mu)\to \frac{1}{N}S^{N/p}\ \ \text{and}
\ \ |\nabla u_{c,\mu}|_p^p\to S^{N/p}
\ \ \text{as}\ \mu\to 0^+.$$
\end{itemize}
\end{theorem}

For $u\in E$ and $t\in \R$, we define
$$u^t(x)=e^{\frac{N}{p}t}u(e^tx),\ {\rm for\ a.e.}\ x\in \R^N. $$
One can easily check that $|u^t|_p=|u|_p$ for any $t\in\R$.
Then $u^t\in S_c$ if and only if $u\in S_c$. Moreover, the properties of $\cM_{c,\mu}$
are closely related to the behavior of $J_{\mu}$ with respect
to such a dilation. For any $u\in S_c$, we introduce the fiber
map
$$\Psi^{\mu}_u(t)=J_{\mu}(u^t)
=\frac{e^{pt}}{p}\int_{\R^N}|\nabla u|^p
-\frac{e^{ p^\ast t}}{p^\ast}\int_{\R^N}|u|^{p^\ast}
-\mu\frac{e^{\gamma_q q t}}{q}\int_{\R^N}|u|^q.$$
The monotonicity and convexity properties of $\Psi^{\mu}_u$
will affect the structure of $\cM_{c,\mu}$ and, obviously,
$u\in\cM_{c,\mu}$ if and only if $0$ is a critical point of
the function $\Psi^{\mu}_u$. In spirit of this, we shall
decompose $\cM_{c,\mu}$ into three parts
$$\cM_{c,\mu}^+=\left\{u\in S_c:\big(\Psi_u^{\mu}\big)'(0)=0,
\big(\Psi_u^{\mu}\big)''(0)>0\right\},$$
$$\cM_{c,\mu}^0=\left\{u\in S_c:\big(\Psi_u^{\mu}\big)'(0)=0,
\big(\Psi_u^{\mu}\big)''(0)=0\right\},$$
$$\cM_{c,\mu}^-=\left\{u\in S_c:\big(\Psi_u^{\mu}\big)'(0)=0,
\big(\Psi_u^{\mu}\big)''(0)<0\right\}.$$
Clearly,
$\cM_{c,\mu}=\cM_{c,\mu}^+\cup \cM_{c,\mu}^0\cup \cM_{c,\mu}^-$.
We will see from Lemmas \ref{lem4.2}, \ref{lem5.4}, \ref{lem5.10}
 that if $p<q<p+p^2/N$ then $\cM_{c,\mu}^0=\emptyset$;
while if $p+p^2/N\leq q<p^\ast$ then
$\cM_{c,\mu}^+=\cM_{c,\mu}^0=\emptyset$.
\begin{remark}
The numbers $p+p^2/N,$ i.e., $L^p$-critical exponents
for $p$-Laplace equations in
$\R^N$, play important roles in our setting. As one
shall see, the structure of the set $\cM_{c,\mu}$ is completely
different in the four cases: $p<q<p+p^2/N$,
$p+p^2/N\leq q<p^\ast$.
\end{remark}

\begin{remark}
The assumption \eqref{eq1.6} plays important role in the study of the geometry of the constrained functional $J_\mu|_{S_c}$ for the two cases $p < q \leq p + p^2/N$
and $p + p^2/N < q < p^\ast$.
In the latter case $p + p^2/N < q < p^\ast$,
it is remarkable that we can prove
that $\alpha (q)= +\infty$, so that any $c, \mu> 0$ are admissible.
The assumption $p<N^{2/3}$ is used in the cases of $p + p^2/N\leq q < p^\ast$
in order to ensure that the ground state level $m(c, \mu)$ is less than $\frac{1}{N}S^{N/p}$, which is an
essential ingredient in our compactness argument.
\end{remark}

Since the variational setting of the problem \eqref{eq1.1} lacks an ordered Hilbert space
structure which plays an important role in dealing with the problems, we
will encounter several difficulties. The first one is the a.e. convergence
of the gradients needs a justification. Furthermore, Palais-Smale sequences of constraint functional
$J_{\mu}|_{S_c}$ are not a priori bounded and the sequence
of approximate Lagrange multipliers has to be controlled. To
overcome such a difficulty, we shall use Ekeland's variational
principle and the stretched functional
in the spirit of \cite{lj} to construct a special Palais-Smale
sequence which carries additional information related to Poho\v zaev
identity. Finally, the weak limit of a bounded Palais-Smale
sequence could not belong to the constraint $S_c$ and
we cannot directly
prove that the weak limit is a critical point of $J_{\mu}|_{S_c}$,
because the embeddings
$E\hookrightarrow L^p(\R^N)$ and $E\hookrightarrow L^{p^\ast}(\R^N)$
are not compact.

The paper is organized as follows. In Section 2, we present
some preliminary results.
We establish the existence of a special Palais-Smale sequence
for the functional $J_{\mu}|_{S_c}$ and give the compactness
analysis In Section 3. Sections 4 is devoted to the proof of Theorem
\ref{the1.1}. The proofs of  Theorems \ref{the1.2} and \ref{the1.3} are given in Sections 5, while the
asymptotic properties of normalized solutions are given in
Section 6.

{\bf Notations.} The following notations will be used frequently.
\begin{itemize}
\item $\N_+$ denotes the positive integer set.

\item $W^{1,p}(\R^N)$ is
the usual Sobolev space with the norm
$$\|u\|=\left(\int_{\R^N}|\nabla u|^p+|u|^p\right)^{1/p}.$$

\item $W_r^{1,p}(\R^N)=\{u\in W^{1,p}(\R^N) : u(x)=u(|x|)\}$ is
equipped with the standard norm $\|\cdot\|$.

\item The standard norm in $L^p(\Omega)$ is denoted by
$|\cdot|_{p,\,\Omega}$ and by $|\cdot|_p$ if $\Omega=\R^N$.

\item $D^{1,p}:=D^{1,p}(\R^N)$ is the completion of
$C_0^\infty(\R^N)$ with respect to $\|u\|_{D^{1,p}}=|\nabla u|_p$.

\item $o(1)$ means a quantity which tends to 0.

\item  The symbols $\rightarrow$ and $\rightharpoonup$ denote
the strong and weak convergence, respectively.

\item $C$, $C(\cdot)$, $C_j$ stand for positive constants whose
exact values are irrelevant.
\end{itemize}

\section{Preliminaries}

Let us now comment on the critical problem in the whole space,
namely
\begin{equation}\label{eq2.1}
-\Delta_p u=|u|^{p^\ast-2}u\ \ \text{in}\ \R^N,\ \ u\in D^{1,p}(\R^N).
\end{equation}
It follows from \cite{gv}(see also \cite{mfb})
that all the regular radial solutions to \eqref{eq2.1} are given by the
following expression:
$$U_{\va,y}(x)=\Big[\frac{\va^{\frac{1}{p}}\left(N^\frac{1}{p}(\frac{N-p}{p-1})^{\frac{p-1}{p}}\right)}
{\va+|x-y|^\frac{p}{p-1}}\Big]^{\frac{N-p}{p}} \ \ \text{with}\ \va>0\ \text{and}\ y\in\R^N.$$
Note that, by \cite{gt}, it follows that the family of functions given above are minimizers
to
$$S=\inf_{u\in D^{1,p}\backslash\{0\}}\frac{|\nabla u|_p^p}{|u|_{p^\ast}^p}.$$
Let $\eta\in C_0^\infty(\R^N)$ be a radial cut-off function satisfying
$$0\leq\eta \leq1,\ \ \eta=1\ \text{in}\ B_1(0),\ \
\eta=0\ \text{in}\ \R^N\setminus B_2(0).$$
We define $u_\va=\eta U_{\va,0}/|\eta U_{\va,0}|_{p^\ast}$ and
\begin{equation}\label{eq2.2}
v_\va=\frac{c u_\va}{|u_\va|_p}
=a_\va u_\va.
\end{equation}
The following estimates can be deduced by standard arguments
(see \cite{gapa1, dh}): as $\va\to 0^+$,
\begin{equation}\label{eq2.3}
|\nabla u_\va|_p^p=S+O(\va^{\frac{N-p}{p}})
\end{equation}
and
\begin{equation}\label{eq2.4}
|u_\va|_s^s=\left\{\begin{array}{ll}
O(\va^{(N(p-s)+sp)(p-1)/p^2}),& \text{if}\ s>p^\ast(1-\frac{1}{p}),\\
O(\va^{N(p-1)/p^2}|\ln\va|),& \text{if}\ s=p^\ast(1-\frac{1}{p}),\\
O(\va^{s(N-p)/p^2}),& \text{if}\ s<p^\ast(1-\frac{1}{p}).
\end{array}\right.
\end{equation}
Then it is easy to obtain the following lemma.
\begin{lemma}\label{lem2.1}
Let $g_\va:[0,+\infty)\to\R$ be defined by
$$g_\va(t)=\frac{t^p}{p}|\nabla v_\va|_p^p
-\frac{t^{p^\ast}}{p^\ast}|v_\va|_{p^\ast}^{p^\ast}.$$
Then, as $\va\to 0^+$, we deduce
$$\sup_{t>0}g_\va(t)\leq\frac{1}{N}S^{N/p}+O(\va^{\frac{N-p}{p}}).$$
\end{lemma}
\begin{lemma}\cite[Lemma 3.6]{BS}\label{lem2.2}
For $u\in S_c$ and $s\in\R$, the map $\varphi\mapsto \varphi^s$ from $T_uS_c$ to
$T_{u^s}S_c$ is a linear isomorphism with inverse $\psi\mapsto\psi^{-s}$, where
$T_uS_c:=\{\varphi\in S_c:\int_{\R^N}|u|^{p-2}u\varphi=0\}$.
\end{lemma}
\begin{proof}[\bf Proof]
We follow the approach in \cite[Lemma 3.6]{BS}. For $\varphi\in T_uS_c$ and $t>0$, we have
$$\int_{\R^N}|u^t(x)|^{p-2}u^t(x)\varphi^t(x)=\int_{\R^N}e^{Nt}|u(e^tx)|^{p-2}u(e^tx)\varphi(e^tx)=
\int_{\R^N}|u(y)|^{p-2}u(y)\varphi(y)=0,$$
which implies that $\varphi^t\in T_{u^t}S_c$ and the map is well defined. Clearly it is linear. Taking into account
that, for every $t,s>0$ and $w\in E$,
$$(w^{t})^s=e^{\frac{N}{p}(t+s)}w(e^{t+s}x)=w^{t+s},\ w^{0}=w,$$
the result follows.
\end{proof}

\begin{lemma}\label{lem2.3}
The map $(u, t)\in E\times\mathbb{R}\rightarrow u^t\in E$ is
continuous.
\end{lemma}
\begin{proof} Assume that
$u_n\to u$ in $E$ and ${t_n}\to t$ in $\mathbb{R}$.
Observed, for $\varphi\in C_0^\infty(\mathbb{R}^N)$, that
\begin{equation*}
\begin{split}&\int_{\mathbb{R}^N}|\nabla u_n^{t_n}|^{p-2}\nabla u_n^{t_n}\nabla\varphi
+\int_{\mathbb{R}^N}|u_n^{t_n}|^{p-2} u_n^{t_n}\varphi\\
&=e^{\frac{p^2-p-N}{p}t_n}\int_{\mathbb{R}^N}|\nabla u_n(x)|^{p-2}\nabla u_n(x)(\nabla\varphi)(e^{-t_n}x)\\
&\ \ \ +e^{-\frac{Nt_n}{p}}\int_{\mathbb{R}^N}|u_n(x)|^{p-2} u_n(x)\varphi(e^{-t_n}x).
\end{split}
\end{equation*}
By Lebesgue dominated convergence theorem, we conclude that
\begin{equation*}
\begin{split}&\left|\int_{\mathbb{R}^N}|\nabla u_n(x)|^{p-2}\nabla u_n(x)(\nabla\varphi)(e^{-t_n}x)-
|\nabla u(x)|^{p-2}\nabla u(x)(\nabla\varphi)(e^{-t_n}x)\right|\\
&+\left|\int_{\mathbb{R}^N}|\nabla u(x)|^{p-2}\nabla u(x)(\nabla\varphi)(e^{-t_n}x)-
|\nabla u(x)|^{p-2}\nabla u(x)(\nabla\varphi)(e^{-t_n}x)\right|\\
&\leq
e^{\frac{N}{p}t_n}\left||\nabla u_n|^{p-2}\nabla u_n-|\nabla u|^{p-2}\nabla u\right|_{\frac{p}{p-1}}
|\nabla \varphi|_p\\
&\ \ \ +\left|\int_{\mathbb{R}^N}|\nabla u(x)|^{p-2}\nabla u(x)(\nabla\varphi)(e^{-t_n}x)-
|\nabla u(x)|^{p-2}\nabla u(x)(\nabla\varphi)(e^{-t}x)\right|\\
&\to 0
\end{split}
\end{equation*}
as $n\to\infty$. In the same way, we can show that
$$\int_{\mathbb{R}^N}|u_n(x)|^{p-2} u_n(x)\varphi(e^{-t_n}x) \to\int_{\mathbb{R}^N}|u(x)|^{p-2} u(x)\varphi(e^{-t}x),$$
as $n\to\infty$. As a consequence, for $\psi\in E$,
\begin{equation*}
\begin{split}&\int_{\mathbb{R}^N}|\nabla u_n^{t_n}|^{p-2}\nabla u_n^{t_n}\nabla\psi
+\int_{\mathbb{R}^N}|u_n^{t_n}|^{p-2} u_n^{t_n}\psi\\
&\to\int_{\mathbb{R}^N}|\nabla u^{t}|^{p-2}\nabla u^{t}\nabla\psi
+\int_{\mathbb{R}^N}|u^{t}|^{p-2} u^{t}\psi
\end{split}
\end{equation*}
as $n\to\infty$, by employing the fact that $C_0^\infty(\R^N)$ is dense in $E$.
Now, it is easy to check that
$$\|u_n^{t_n}\|^p=e^{pt_n}|\nabla u_n|_p^p+|u_n|_p^p
\to e^{pt}|\nabla u|_p^p+|u|_p^p=\|u^{t}\|^p,$$
as $n\to\infty$. Using a well known inequality found in \cite[Lemma A.0.5]{ip}, we know that
\begin{equation}\label{eq2.5}(|\eta|^{p-2}\eta-|\xi|^{p-2}\xi)\cdot(\eta-\xi)\geq\left\{\begin{array}{ll}
d_1|\eta-\xi|^p,&{\rm if}\ p\geq 2,\\ d_2(|\xi|+|\eta|)^{p-2}|\xi-\eta|^2,&{\rm if}\ p\in(1,2),
\end{array}\right.\end{equation}
where $d_1,d_2$ are positive constants.
Thereby, we infer from \eqref{eq2.5} that $u_n^{t_n}\to u^t$ in $E$.
The proof is completed.\end{proof}

Finally, we give a version of linking theorem, see \cite[Section 5]{ng}.

\begin{definition}
Let $X$ be a topological space and $B$ be a closed subset of $X$. We shall say that a class ${\mathcal F}$
of compact of subsets of $X$ is a homotopy-stable family with extended boundary $B$ if for any set $A$ in
${\mathcal F}$ and any $\eta\in C([0,1]\times X; X)$ satisfying $\eta(t,x)=x$ for all $(t,x)\in(\{0\}\times X)\cup
([0,1]\times B)$ we have that $\eta(\{1\}\times A)\in{\mathcal F}$.
\end{definition}

\begin{lemma}\cite[Theorem 5.2]{ng}\label{lem2.4}
Let $\phi$ be a $C^1$-functional on a complete connected $C^1$-Finsler manifold $X$ and consider a
homotopy-stable family ${\mathcal F}$ with an extended closed boundary $B$. Set $m=m(\phi,\mathcal F)$
and let $F$ be a closed subset of $X$ satisfying

(1) $(A\cap F)\setminus B\neq\emptyset$ for every $A\in{\mathcal F}$;

(2) $\sup\phi(B)\leq m\leq\inf\phi(F)$.\\
Then, for any sequence of sets $(A_n)_n$ in ${\mathcal F}$ such that $\lim_{n\to\infty}\sup_{A_n}\phi=m$, there
exists a sequence $(x_n)_n$ in $X$ such that
$$\lim_{n\to\infty}\phi(x_n)=m,\ \lim_{n\to\infty}\|d\phi(x_n)\|=0,\ \lim_{n\to\infty}dist(x_n,F)=0,\ \lim_{n\to\infty}dist(x_n,A_n)=0.$$
\end{lemma}

\section{Properties of Palais-Smale sequences \label{Section 3}}
For simplicity of notations, we will write $J_{\mu}$,
$P_{\mu}$, $\Psi_u^{\mu}$, $\cM_{c,\mu}$,
$\cM_{c,\mu}^\pm$, $\cM_{c,\mu}^0$, $m(c,\mu)$ and
$m^\ast(c,\mu)$ as $J$, $P$, $\Psi_u$, $\cM$, $\cM^\pm$,
$\cM^0$, $m$ (or $m(c)$) and $m^\ast$ respectively in
Sections 3, 4 and 5.

\begin{lemma}\label{lem3.1}
Let $p<q<p^\ast$ and $c,\ \mu>0$. Assume that $\{u_n\}\subset S_c$
be a Palais-Smale sequence for $J|_{S_c}$ at level $m\in \R$ with
\begin{equation}\label{eq3.1}
P(u_n)\to 0,\ \ \text{as}\ n\to\infty.
\end{equation} Then the sequence $\{u_n\}$ is bounded in $E$.
\end{lemma}

\begin{proof}[\bf Proof]
If $p<q<p+p^2/N$, then $0\le\gamma_q q<p$. The fact that $P(u_n)=o(1)$ and the Gagliardo-Nirenberg inequality lead to
\begin{align*}
C&\geq J(u_n)-\frac{1}{p^\ast}P(u_n)\\
&=\frac{1}{N}|\nabla u_n|_p^p-\frac{\mu}{q}\left(1-\frac{\gamma_q q}{p^\ast}\right)|u_n|_q^q\\
&\geq\frac{1}{N}|\nabla u_n|_p^p
-\frac{\mu}{q}C(q)\left(1-\frac{\gamma_q q}{p^\ast}\right)c^{(1-\gamma_q)q}|\nabla u_n|_p^{\gamma_q q}
\end{align*}
provided that $n$ is sufficiently large, from which we infer that $\{u_n\}$ is bounded in $E$.

If $p+p^2/N\leq q<p^\ast$, then $p\le\gamma_q q<p^\ast$. Using $P(u_n)=o(1)$ again implies
\begin{align*}
C&\geq J(u_n)-\frac{1}{p}P(u_n)\\
&=\frac{1}{N}|u_n|_{p^\ast}^{p^\ast}
+\frac{\mu(\gamma_qq-p)}{pq}|u_n|_q^q\\
&\geq\frac{1}{N}|u_n|_{p^\ast}^{p^\ast}
\end{align*}
for $n$ large.
The H\"{o}lder inequality ensures that $|u_n|_q^q\leq|u_n|_p^{p(1-\tau)}|u_n|_{p^\ast}^{p^\ast\tau}\leq C$
for suitable $\tau\in(0,1)$,
which together with $P(u_n)=o(1)$ yields that
\begin{align*}
|\nabla u_n|_p^p&=|u_n|_{p^\ast}^{p^\ast}+\mu \gamma _q|u_n|_q^q+o(1).
\end{align*}
Hence $\{u_n\}$ is also bounded in $E$.
\end{proof}

In order to show the convergence
a.e. of the gradients, we introduce the following lemma.
\begin{lemma}[\cite{clr}] \label{lem3.2}
Let $\{v_k\}$ be a sequence in $D^{1,p}(\R^N)$  and be such that $v_k\rightharpoonup v$ weakly in
$D^{1,p}(\R^N)$. Assume that, for every $\psi\in C_0^\infty(\R^N)$,
$$\lim_{k\to\infty}\int_{\R^N}\psi\left(|\nabla v_k|^{p-2}\nabla v_k-|\nabla v|^{p-2}\nabla v\right)\cdot\nabla(T(v_k-v))=0,$$
where the truncation function $T:\R\to[0,1]$ is given by
 \begin{equation}\label{eq3.2}T(t)=\left\{\begin{array}{ll}
t,&{\rm if}\ |t|\leq 1,\\ \frac{t}{|t|},&{\rm if}\ |t|\geq1 .
\end{array}\right.\end{equation}
Then, after passing to a subsequence, $\nabla v_k\to \nabla v$ a.e. in $\R^N$.
\end{lemma}

\begin{lemma}\label{lem3.3}
Let $0\neq m<\frac{1}{N}S^{N/p}$. If $\{u_n\}\subset S_c$ is a Palais-Smale
sequence for $J|_{S_c}$ at the level $m$ and such that \eqref{eq3.1}
holds, then up to a subsequence one of the following alternatives
holds:\\
$(i)$ $u_n\rightharpoonup u$ in $E$ for some $u\neq0$ and
$$J(u)\leq m-\frac{1}{N}S^{N/p};$$
$(ii)$ $u_n\to u$ in $E$ for some $u$, $J(u)=m$
and $u$ solves \eqref{eq1.1} and \eqref{eq1.3} for some $\lambda<0$.
\end{lemma}

\begin{proof}[\bf Proof]

We obtain from Lemma \ref{lem3.1} that $\{u_n\}$ is bounded in $E$. Then, going
if necessary to a subsequence, we may assume that $u_n\rightharpoonup u$
in $E$, $u_n\to u$ in $L^q(\R^N)$ for $p<q<p^\ast$ and $u_n(x)\to u(x)$ a.e. $x\in\R^N$.
Since $(J|_{S_c})'(u_n)\to0$, there exists a sequence $\{\lambda_n\}\subset\R$ such that for $\varphi\in E$,
\begin{align}\label{eq3.3}
\int_{\R^N}|\nabla u_n|^{p-2}\nabla u_n\nabla \varphi
&=\int_{\R^N}|u_n|^{p^\ast-2}u_n\varphi+\mu\int_{\R^N}|u_n|^{q-2}u_n\varphi
\notag\\
&\ \ \ +\lambda_n\int_{\R^N}|u_n|^{p-2}u_n\varphi+o(1)\|\varphi\|
\end{align}
as $n\to\infty$.
Choosing $u_n$ as a test function in \eqref{eq3.3},
it is easy to verify that $\{\lambda_n\}$ is bounded. We may therefore assume, up to a subsequence,
that $\lambda_n\to\lambda$ as $n\to\infty$. Then, in view of
\eqref{eq3.1}, \eqref{eq3.3} and $\gamma_q<1$, we find
\begin{align*}
\lambda c^p&=\lim_{n\to\infty}\lambda_n|u_n|_p^p\\
&=\lim_{n\to\infty}\left(|\nabla u_n|_p^p
-|u_n|_{p^\ast}^{p^\ast}-\mu|u_n|_q^q\right)\\
&=\lim_{n\to\infty}\mu(\gamma_q-1)|u_n|_q^q\\
&=\mu(\gamma_q-1)|u|_q^q\leq 0.
\end{align*}
Now, we claim that $u\neq0$. To prove the claim, we argue by contradiction that $u=0$, then combining $u_n\to 0$ in $L^q(\R^N)$ with the fact that $P(u_n)=o(1)$ as $n\to\infty$ implies
\begin{equation}\label{eq3.4}
|\nabla u_n|_p^p-\int_{\mathbb{R}^N}|u_n|^{p^\ast}
=o(1)\end{equation}
as $n\to\infty$.
Going if necessary to a subsequence, we may assume that
$|\nabla u_n|_p^p\to l$, as $n\to\infty$.
On the other hand, the estimate \eqref{eq3.4} and the Sobolev inequality lead to
$$l\leq S^{-\frac{N}{N-p}}l^\frac{N}{N-p}.$$
Thus, it is easy to obtain either $l=0$ or $l\geq S^\frac{N}{p}$.
On the other hand,
$$m=J(u_n)-\frac{1}{p^\ast}P(u_n)+o(1)
=\frac{1}{N}|\nabla u_n|_p^p+o(1)
=\frac{1}{N}l.$$
Now we also obtain either
$m=0$ or $m\geq \frac{1}{N}S^{N/p}$.
This is a contradiction with the assumption $0\neq m<\frac{S^{N/p}}{N}$. Consequently, $u\neq 0$ and henceforth $\lambda<0$.
In the following, we will show that
\begin{equation}\label{eq3.5}\nabla u_n(x)\to \nabla u(x)\ {\text a.e.\ in}\ \R^N.\end{equation}
Actually, similar to \cite{clr}, choose $\psi\in C_0^\infty(\R^N)$, the Egorov's Theorem implies that
for every $\delta>0$ there exists $E_\delta\subset supp(\psi)$ such that $|E_\delta|<\delta$ and
$u_n\to u$ uniformly in $supp(\psi)\setminus E_\delta$. Hence, $|u_n(x)-u(x)|\leq 1$ for all $x\in supp(\psi)\setminus E_\delta$ and $n$ large enough. Then, from \eqref{eq3.2}, we can assert that
 \begin{equation*}
\begin{split}
&\left|\int_{\R^N}\psi|\nabla u|^{p-2}\nabla u\cdot \nabla(T(u_n-u))\right|\\
&\leq\left|\int_{\R^N\setminus E_\delta}\psi|\nabla u|^{p-2}\nabla u\cdot \nabla(u_n-u)\right|
+\left|\int_{\R^N\setminus E_\delta}\psi|\nabla u|^{p-2}\nabla u\cdot \nabla(T(u_n-u))\right|\\
&\leq o(1)+C\delta.
\end{split}
\end{equation*}
Thus, we infer that
\begin{equation}\label{eq3.6}\lim_{n\to\infty}\int_{\R^N}\psi|\nabla u|^{p-2}\nabla u\cdot \nabla(T(u_n-u))=0.
\end{equation}
On the other hand, according to H\"{o}lder's inequality and the dominated convergence theorem, we can demonstrate that
 \begin{equation*}
\begin{split}
&\left|\int_{\R^N}\psi|\nabla u_n|^{p-2}\nabla u_n\cdot \nabla(T(u_n-u))\right|\\
&\leq\left|\int_{\R^N}|\nabla u_n|^{p-2}\nabla u_n\cdot \nabla(\psi T(u_n-u))\right|
+\left|\int_{\R^N}|\nabla u_n|^{p-2}\nabla u_n\cdot(T(u_n-u))\nabla\psi\right|\\
&\leq\left|\int_{\R^N}|u_n|^{p^\ast-2}u_n(\psi T(u_n-u))\right|+\mu\left|\int_{\R^N}|u_n|^{q-2}u_n(\psi T(u_n-u))\right|
\notag\\
&\ \ \ +\left|\int_{\R^N}\lambda_n|u_n|^{p-2}u_n(\psi T(u_n-u))\right|+\left|\int_{\R^N}|\nabla u_n|^{p-2}\nabla u_n\cdot(T(u_n-u))\nabla\psi\right|+o(1)\\
&\leq C\left(\int_{\R^N}|\psi T(u_n-u)|^{p^\ast}\right)^{\frac{1}{p^\ast}}+ C\left(\int_{\R^N}|\psi T(u_n-u)|^{q}\right)^{\frac{1}{q}}\\
\notag
&\ \ \ +C\left(\int_{\R^N}|\psi T(u_n-u)|^{p}\right)^{\frac{1}{p}}+C\left(\int_{\R^N}|T(u_n-u)\nabla\psi |^{p}\right)^{\frac{1}{p}}+o(1)\\
&= o(1).
\end{split}
\end{equation*}
This means that
$$\lim_{n\to\infty}\left|\int_{\R^N}\psi|\nabla u_n|^{p-2}\nabla u_n\cdot \nabla(T(u_n-u))\right|=0,$$
which together with \eqref{eq3.6} and Lemma \ref{lem3.2} implies that \eqref{eq3.5} holds.
Letting $n\to\infty$ in \eqref{eq3.3}, we have
\begin{equation}\label{eq3.7}
\int_{\R^N}|\nabla u|^{p-2}\nabla u\nabla \varphi
=\int_{\R^N}|u|^{p^\ast-2}u\varphi+\mu\int_{\R^N}|u|^{q-2}u\varphi
+\lambda\int_{\R^N}|u|^{p-2}u\varphi,
\end{equation}
which means that $u$ is a weak solution of the semilinear equation
$$-\Delta_p u=\lambda |u|^{p-2}u+\mu|u|^{q-2}u+|u|^{p^\ast-2}u\ \
\text{in}\ \R^N.$$
Thanks to \cite[Proposition 2.1]{ll11}, we infer that $u$ satisfies the following Poho\v zaev
identity
\begin{equation}\label{eq3.8}
\int_{\R^N}|\nabla u|^p
-\int_{\R^N}|u|^{p^\ast}-\mu \gamma_q \int_{\R^N}|u|^q=0.
\end{equation}
Let $v_n=u_n-u$ and then the Br\'{e}zis-Lieb Lemma \cite[Lemma 1.32]{w} leads to
$$|\nabla u_n|_p^p=|\nabla v_n|_p^p+|\nabla u|_p^p+o(1), \ |u_n|_{p^\ast}^{p^\ast}=|v_n|_{p^\ast}^{p^\ast}+|u|_{p^\ast}^{p^\ast}+o(1)$$
and
$$|u_n|_q^q=|v_n|_q^q+|u|_q^q+o(1),$$
from which it is reduced to
\begin{equation*}
\begin{split}
P(u_n)&=|\nabla u_n|_p^p
-|u_n|_{p^\ast}^{p^\ast}-\mu \gamma _q|u_n|_q^q\\
&=|\nabla v_n|_p^p-|v_n|_{p^\ast}^{p^\ast}\\
&\quad\,+|\nabla u|_p^p
-|u|_{p^\ast}^{p^\ast}-\mu \gamma _q|u|_q^q+o(1).
\end{split}
\end{equation*}
Combining this with \eqref{eq3.1}, \eqref{eq3.8} guarantees that
\begin{equation}\label{eq3.9}
|\nabla v_n|_p^p
=|v_n|_{p^\ast}^{p^\ast}+o(1)\leq S^{-\frac{p^\ast}{p}}|\nabla v_n|_p^{p^\ast}+o(1).
\end{equation}
Suppose that $\lim_{n\to\infty}|\nabla v_n|_p^p=l_1$. It is easy to check from \eqref{eq3.9} that
$$l_1\leq S^{-\frac{N}{N-p}}l_1^\frac{N}{N-p}.$$
Similarly, either $l_1=0$ or $l_1\geq S^\frac{N}{p}$.
Then, we consider the two cases which is described as follows.

{\bf Case 1.} $l_1\geq S^\frac{N}{p}$. In this
case, we can easily check that
\begin{align*}
m&=J(u_n)+o(1)\\
&=J(u)+\frac{1}{p}|\nabla v_n|_p^p
-\frac{1}{p^\ast}|v_n|_{p^\ast}^{p^\ast}+o(1)\\
&=J(u)+\frac{1}{N}|\nabla v_n|_p^p+o(1)\\
&\geq J(u)+\frac{1}{N}S^\frac{N}{p},
\end{align*}
where \eqref{eq3.9} is used.

{\bf Case 2.} $l=0$. In this case, it is easy to see $u_n\to u$ in
$D^{1, p}(\R^N)$ and hence in $L^{p^\ast}(\R^N)$ by the Sobolev
inequality. We deduce, by taking $u_n-u$ as a test function in \eqref{eq3.3}
and \eqref{eq3.7} respectively and subtracting them, that
\begin{align*}
&\int_{\R^N}(|\nabla u_n|^{p-2}\nabla u_n-|\nabla u|^{p-2}\nabla u)\cdot\nabla(u_n-u)
=\int_{\R^N}(|u_n|^{p^\ast-2}u_n-|u|^{p^\ast-2}u)(u_n-u)\\
&\ \ \ \quad\,+\mu\int_{\R^N}(|u_n|^{q-2}u_n-|u|^{q-2}u)(u_n-u)
+\int_{\R^N}(\lambda_n|u_n|^{p-2}u_n-\lambda |u|^{p-2}u)(u_n-u)\\
&\ \ \ \quad\,
+o(1)\|u_n-u\|,
\end{align*}
which ensures that
\begin{align}\label{eq3.10}0&=\lim_{n\to\infty}\int_{\R^N}(\lambda_n|u_n|^{p-2}u_n-\lambda |u|^{p-2}u)(u_n-u)\notag\\
&=\lim_{n\to\infty}\lambda\int_{\R^N}(|u_n|^{p-2}u_n- |u|^{p-2}u)(u_n-u).\end{align}
For $1<p<2$, we deduce from \eqref{eq3.10} and \eqref{eq2.5} that
\begin{align*}
\left(\int_{\R^N}|u_n-u|^p\right)^{\frac{2}{p}}
\notag&\leq\left(\int_{\R^N}\frac{|u_n-u|^2}{(|u_n|+|u|)^{2-p}}\right)
\left(\int_{\R^N}(|u_n|+|u|)^{p}\right)^{\frac{2-p}{p}}\\
&\leq C\int_{\R^N}(|u_n|^{p-2}u_n- |u|^{p-2}u)(u_n-u)
\to 0.
\end{align*}
Similarly, we can prove the same convergence property for the case $p\geq 2$.
The above limits lead to $u_n\to u$ in $L^p(\R^N)$ and then in $E$.
Hence, the alternative $(ii)$ is valid.
\end{proof}

\section{$L^p$-subcritical perturbation}
This section is devoted to the proof of Theorem \ref{the1.1}.
We always assume that
$$p<q<p+p^2/N,\ \ c,\ \mu>0,\ \mu c^{(1-\gamma_q)q}<\alpha(q).$$

\subsection{Ground state solution}
Let $p<q<p+p^2/N$, which implies that $0<\gamma_qq<p$. Furthermore, taking into account Lemma
\ref{lem1.1}, we see that
\begin{align}\label{eq4.1}
J(u)&=\frac{1}{p} |\nabla u|_p^p
-\frac{1}{p^\ast}|u|_{p^\ast}^{p^\ast}-\frac{\mu}{q}|u|_q^q\notag\\
&\geq\frac{1}{p} |\nabla u|_p^p
-\frac{S^{-\frac{p^\ast}{p}}}{p^\ast}|\nabla u|_p^{p^\ast}
-\frac{\mu}{q}C(q)c^{(1-\gamma_q)q}|\nabla u|_p^{\gamma_q q}
\end{align}
for $u\in S_c$. In what follows, we study the properties of the fiber
map $\Psi_u^\mu$,
which need the following technical lemma.

\begin{lemma}\label{lem4.1}
The function
$$h(t)=\frac{1}{p}t^p-\frac{S^{-\frac{p^\ast}{p}}}{p^\ast}t^{p^\ast}
-\frac{\mu}{q}C(q)c^{(1-\gamma_q)q}t^{\gamma_qq},\ \ t>0$$
has a strict local minimum with negative value and a strict
global maximum with positive value. Furthermore, there exist two
positive numbers $R_0<R_1$ depending on $c$ and $\mu$ such that
$h(R_0)=h(R_1)=0$ and $h(t)>0$ if and only if $t\in(R_0,R_1)$.
\end{lemma}

\begin{proof}[\bf Proof]
We first consider the auxiliary function $\varphi: \R^+\to\R$ given by
$$\varphi(t):=\varphi_c(t)=\frac{1}{p}t^{p-\gamma_q q}
-\frac{S^{-\frac{p^\ast}{p}}}{p^\ast}t^{p^\ast-\gamma_q q}
-\frac{\mu}{q}C(q)c^{(1-\gamma_q)q}.$$
By carrying out a straightforward computation, we conclude that $\varphi$ achieves its maximum at
$$t^\ast=\Big(\frac{p^\ast(p-\gamma_q q)S^{\frac{p^\ast}{p}}}
{p(p^\ast-\gamma_q q)}\Big)^{\frac{1}{p^\ast-p}},$$
and the maximum value of $\varphi$ is
\begin{align*}\varphi(t^\ast)&=\left(\frac{p^\ast(p-\gamma_q q)S^{\frac{p^\ast}{p}}}
{p(p^\ast-\gamma_q q)}\right)^{\frac{p-\gamma_q q}{p^\ast-p}}\frac{p^\ast-p}{p(p^\ast-\gamma_q q)}-\frac{\mu}{q}C(q)c^{(1-\gamma_q)q}.
\end{align*}
Thus, recalling that
$$\mu c^{(1-\gamma_q)q}<\alpha(q)=\left(\frac{p^\ast(p-\gamma_q q)S^{\frac{p^\ast}{p}}}
{p(p^\ast-\gamma_q q)}\right)^{\frac{p-\gamma_q q}{p^\ast-p}}\frac{(p^\ast-p)q}{C(q)p(p^\ast-\gamma_q q)},$$
we conclude that $h(t^\ast)=(t^\ast)^{\gamma_qq}\varphi(t^\ast)>0$. This, together with
$h(t)\to0^-$ as $t\to 0^+$ and $h(t)\to-\infty$ as $t\to +\infty$
implies that $h$ has at least two zero points $0<R_0<R_1$.
Clearly $\varphi$
has only one critical point.
Notice in particular that, since $h(t)=0$ is equivalent to $\varphi(t)=0$, $h$ has exactly two zero
points $0<R_0<R_1$ and $h(t)>0$ if and only if $t\in(R_0,R_1)$.

Consequently, $h$ achieves
its global maximum with positive value in $(R_0,R_1)$ and $h$
achieves a local minimum with negative value in $(0,R_0)$. Observe
that $h'(t)=0$ is equivalent to
$$\psi(t)=t^{p-\gamma_q q}
-S^{-\frac{p^\ast}{p}}t^{p^\ast-\gamma_q q}-\mu\gamma_qC(q)c^{(1-\gamma_q)q}=0.$$
We deduce from $\psi$ has only one critical point in $(0,+\infty)$ that $h$ has exactly two critical points, which are the
strict local minimum point and the global maximum point previously
found.
\end{proof}

\begin{lemma}\label{lem4.2}
$\cM^0=\emptyset$.
\end{lemma}

\begin{proof}[\bf Proof]
Suppose the thesis is false. Thus, there exists
$u\in\cM^0$ such that

\begin{equation}\label{eq4.2}
|\nabla u|_p^p-|u|_{p^\ast}^{p^\ast}-\mu \gamma _q|u|_q^q=0
\end{equation}
and
\begin{equation}\label{eq4.3}
p|\nabla u|_p^p
-p^\ast|u|_{p^\ast}^{p^\ast}-\mu \gamma_q^2 q|u|_q^q=0.
\end{equation}
Combined \eqref{eq4.2} with \eqref{eq4.3}, we deduce that
\begin{align*}(p-\gamma_q q)|\nabla u|_p^p
&=(p^\ast-\gamma_q q)|u|_{p^\ast}^{p^\ast}\\
&\leq (p^\ast-\gamma_q q)S^{-\frac{p^\ast}{p}}|\nabla u|_p^{p^\ast},
\end{align*}
and a simple computation shows that
$$|\nabla u|_p^{p-\gamma_q q}
\geq\left(\frac{S^{\frac{p^\ast}{p}}(p-\gamma_q q)}{p^\ast-\gamma_q q}\right)^{\frac{p-\gamma_q q}{p^\ast-p}}.$$
The Lemma \ref{lem1.1} combined with \eqref{eq4.2}$-$\eqref{eq4.3} infers that
\begin{align*}
(p^\ast-p)|\nabla u|_p^p
=\mu \gamma_q(p^\ast-\gamma_q q)|u|_q^q
\leq(p^\ast-\gamma_q q)\gamma_q C(q)\mu c^{(1-\gamma_q)q}|\nabla u|_p^{\gamma_q q},
\end{align*}
which therefore implies that
$$|\nabla u|_p^{p-\gamma_q q}
\leq \frac{(p^\ast-\gamma_q q)\gamma_q C(q)\mu c^{(1-\gamma_q)q}}{p^\ast-p}.$$
Hence, recalling that $0<\gamma_qq<p$, we conclude that
$$\mu c^{(1-\gamma_q)q}
\geq\left(\frac{S^{\frac{p^\ast}{p}}(p-\gamma_q q)}{p^\ast-\gamma_q q}\right)^{\frac{p-\gamma_q q}{p^\ast-p}}
\frac{p^\ast-p}{\gamma_q C(q)(p^\ast-\gamma_q q)}=C'\cdot\frac{p}{\gamma_q q}\cdot\left(\frac{p^\ast}{p}\right)^{-\frac{p-\gamma_q q}{p^\ast-p}}>C',$$
which is a contradiction with the assumption
$\mu c^{(1-\gamma_q)q}<C'$.
\end{proof}

\begin{lemma}\label{lem4.3}
For every $u\in S_c $, the function $\Psi_u$ has exactly two
critical points $s_u,t_u$ and two zeros $c_u,d_u$ such that
$s_u<c_u<t_u<d_u$. Moreover, we have\\
$(i)$ $u^{s_u}\in \cM^+$ and $u^{t_u}\in \cM^-$;\\
$(ii)$ $|\nabla u^s|_p< R_0$ for $s< c_u$ and
\begin{equation}\label{eq4.4}
J(u^{s_u})
=\min\left\{J(u^s) : s\in \R,\ |\nabla u^s|_p<R_0\right\}
<0,
\end{equation}
where $R_0$ is given in Lemma \ref{lem4.1};\\
$(iii)$ there holds
\begin{equation}\label{eq4.5}
J(u^{t_u})=\max\left\{J(u^s): s\in \R\right\}>0
\end{equation}
and $\Psi_u$ is decreasing on $(t_u,\infty)$, henceforth if
$t_u<0$ then $P(u)<0$;\\
$(iv)$ the maps $u\in S_c\mapsto s_u\in\R$ and
$u\in S_c\mapsto t_u\in\R$ are of class $C^1$.
\end{lemma}

\begin{proof}[\bf Proof]
Let $u\in S_c$ and recall that, by \eqref{eq4.1},
$$\Psi_u(t)=J(u^t)\geq h(|\nabla u^t|_2)=h(e^t|\nabla u|_2).$$
According to Lemma \ref{lem4.1}, the function $\Psi_u$ is positive
on $(\log(R_0/ |\nabla u|_2),\log(R_1/ |\nabla u|_2))$ which
combined with the facts $\Psi_u(t)\to0^-$ as $t\to-\infty$ and
$\Psi_u(t)\to-\infty$ as $t\to+\infty$ indicates that $\Psi_u$
has at least two critical points $s_u<t_u$. Here $s_u$ is the
local minimum point of $\Psi_u$ on $(-\infty,\log(R_0/ |\nabla u|_2))$
with $\Psi_u(s_u)<0$, while $t_u$ is the global maximum point
of $\Psi_u$ with $\Psi_u(t_u)>0$. Note that $\Psi_u'(t)=0$ is
equivalent to
$$\phi_u(t)=e^{(p-\gamma_q q)t}|\nabla u|_p^p
-e^{(p^\ast-\gamma_q q) t}|u|_{p^\ast}^{p^\ast}
-\mu \gamma_q|u|_q^q=0$$
and clearly $\phi_u$ has a unique critical point. Therefore, this means that$\Psi_u$
has exactly two critical points $s_u$ and $t_u$. Using a similar
argument, one can prove that $\Psi_u$ also has exactly two zeros
$c_u$ and $d_u$ with $s_u<c_u<t_u<d_u$.

Recall that $u^t\in \cM$ if and only if $\Psi_u'(t)=0$. Since
$s_u$ is the local minimum point of $\Psi_u$ and $t_u$ is the
global maximum point of $\Psi_u$, we have
$\Psi_u'(s_u)=\Psi_u'(t_u)=0$ and $\Psi_u''(s_u)\geq0\geq\Psi_u''(t_u)$.
Then Lemma \ref{lem4.2} ensures that $u^{s_u}\in \cM^+$
and $u^{t_u}\in\cM^-$.

Clearly, $|\nabla u^s|_2< |\nabla u^{c_u}|_2\leq R_0$ if $s< c_u$,
\eqref{eq4.4} is a consequence of the fact that $s_u$ is local
minimum point of $\Psi_u$ on $(-\infty,\log(R_0/ |\nabla u|_2))$
with $\Psi_u(s_u)<0$, and \eqref{eq4.5} is a consequence of the
fact that $t_u$ is the global maximum point of $\Psi_u$ with
$\Psi_u(t_u)>0$.

Observe that $\Psi'_u$ has exactly two zeros $s_u$ and $t_u$.
Moreover, $\Psi'_u(t)<0$ if $t>t_u$ which means that $\Psi_u$
is decreasing on $(t_u,+\infty)$. In particular, if $t_u<0$
then $P(u)=\Psi_u'(0)<0$.

Next we will show $u\mapsto t_u$ is of class $C^1$. For the purpose, let us define
$\Phi: S_c\times\R\to\R$ by
$\Phi(u,t)=\Psi_u'(t)$. It is clear that $\Phi$ is of class
$C^1$, $\Phi(u, t_u)=0$ and $\partial_t\Phi(u,t_u)=\Psi''_u(t_u)<0$.
Applying the Implicit Function Theorem, we see that the map
$u\mapsto t_u$ is of class $C^1$. The same argument proves
that $u\mapsto s_u$ is also $C^1$.
\end{proof}

It is easy to check that $R_0< t^\ast\leq S^{\frac{p^\ast}{p(p^\ast-p)}}$, where $t^\ast$ given in Lemma \ref{lem4.1}.
Recall that
$$D_{k}(c)=\{u\in S_c : |\nabla u|_2<k\},\ \ \text{for}\ k>0.$$
Thus, $D_{R_0}(c)\subset A_{t^\ast}(c)$. We deal in the sequel with the functional $J|_{S_c}$ has a positive
ground state in $D_{R_0}(c)$. For this purpose, we consider the infimum
$$\inf_{u\in D_{R_0}(c)}J(u).$$

\begin{lemma}\label{lem4.4}
We have
$$m(c)=\inf_{u\in\cM^+}J(u)=\inf_{u\in D_{R_0}(c)}J(u)=\inf_{u\in D_{t^\ast}(c)}J(u)
<\inf_{u\in\,\overline{D_{R_0}(c)}\setminus D_{R_0-\delta}(c)}J(u)$$
for $\delta>0$ small and $m(c)\in(-\infty,0)$.
\end{lemma}

\begin{proof}[\bf Proof]
We deduce from Lemma \ref{lem4.3} that $J(u)<0$ if $u\in\cM^+$ while $J(u)>0$ if
$u\in\cM^-$, which together with Lemma \ref{lem4.2} implied $\inf_{u\in\cM^+}J(u)=\inf_{u\in\cM}J(u)=m(c)$.
It follows from $\cM^+\subset D_{R_0}(c)$ that $\inf_{u\in D_{R_0}(c)}J(u)\leq m(c)$.
By Lemma \ref{lem4.3} again, for $u\in D_{R_0}(c)$, we have
$u^{s_u}\in \cM^+$ and
$$J(u^{s_u})=\min\left\{J(u^s) : s\in \R,\ |\nabla u^s|_2<R_0\right\}
\leq J(u),$$
which implies $m(c)\leq \inf_{u\in D_{R_0}(c)}J(u)$ and henceforth
$m(c)=\inf_{u\in D_{R_0}(c)}J(u)$.

It follows from \eqref{eq4.1} that, for $u\in D_{R_0}(c)$,
$$J(u)\geq h(|\nabla u|_2)\geq\min_{t\in[0,R_0]}h(t)>-\infty.$$
Lemma \ref{lem4.3} states that $u^{s_u}\in D_{R_0}(c)$ and
$J(u^{s_u})<0$ if $u\in S_c$. Then $-\infty<m(c)<0$.

In view of Lemma \ref{lem4.1}, $J(u)\geq 0$ for
$u\in D_{t^\ast}(c)\setminus D_{R_0}(c)$ which implies
$m(c)=\inf_{u\in D_{t^\ast}(c)}J(u)$.
Since $h$ is continuous and $h(R_0)=0$, there holds
$$\inf_{t\in[R_0-\delta,R_0]}h(t)>m(c)\ \
\text{for}\ \delta>0\ \text{small}.$$
Note that $J(u)\geq h(|\nabla u|_2)$ for $u\in S_c$. Then we
have
$$\inf_{u\in\,\overline{D_{R_0}(c)}\setminus D_{R_0-\delta}(c)}J(u)
\geq \inf_{t\in[R_0-\delta,R_0]}h(t)>m(c),$$
concluding the proof.
\end{proof}

For $0<R_0<R_1<\infty$, let $\xi:\R^+\to[0,1]$ be a nonincreasing and $C^\infty$ function satisfying
$\xi(x)=1$, $x\in[0,R_0]$ and $\xi(x)=0$, $x\geq R_1$. In what follows, let us consider the truncated
functional
\begin{align*}
J_T(u)
&=\frac{1}{p} |\nabla u|_p^p
-\frac{\xi(|\nabla u|_p)}{{p^\ast}}|u|_{p^\ast}^{p^\ast}-\frac{\mu}{q}|u|_q^q\\
&\geq\frac{1}{p} |\nabla u|_p^p
-\frac{S^{-\frac{p^\ast}{p}}\xi(|\nabla u|_p)}{p^\ast}|\nabla u|_p^{p^\ast}
-\frac{\mu}{q}C(q)c^{(1-\gamma_q)q}|\nabla u|_p^{\gamma_q q}.
\end{align*}

A direct computation shows $J_T(u)\geq \bar{h}(|\nabla u|_p)$, where $\bar{h}: \mathbb{R}^+\to\mathbb{R}$ is given by $$\bar{h}(t)=\frac{1}{p} t^p
-\frac{S^{-\frac{p^\ast}{p}}\xi(t)}{p^\ast}t^{p^\ast}
-\frac{\mu}{q}C(q)c^{(1-\gamma_q)q}t^{\gamma_q q}.$$
Observe that $\bar{h}<0$ over interval $(0,R_0)$ and $\bar{h}(t)>0$ over $(R_0,+\infty)$ and hence
$$m(c):=m(c, \mu)=\inf_{u\in D_{R_0}(c)}J(u)=\inf_{u\in D_{R_0}(c)}J_T(u)=\inf_{u\in S_c}J_T(u).$$

\begin{lemma}\label{lem4.5}
$(i)$ $J_T\in C^1(E,\R)$. If $J_T(u)<0$ then $|\nabla u|_p<R_0$ and $J(v)=J_T(v)$ for all
$v$ in a small enough neighborhood of $u$.\\
$(ii)$ $J_T$ verifies a local $(PS)_m$ condition on $S_c$ for the level $m<0$.
\end{lemma}

\begin{proof}[\bf Proof]
Conclusion (i) follows immediately from the definition of $J_T$.

(ii) Let $\{u_n\}$ be a $(PS)_m$ sequence of $J_T$ prescribed to $S_c$ with $m<0$. Then
by the definition of $J_T$, we have $|\nabla u_n|_p<R_0<t_0$ for $n$ large enough. And hence
$\{u_n\}$ is also a $(PS)_m$ sequence of $J$ prescribed to $S_c$ with $m<0$.
Since $\{u_n\}\subset E$ is a bounded, we may assume that $u_n\rightharpoonup u$ weakly in $E$ and $u_n\rightarrow u$ strongly in $L^{q}(\mathbb{R}^N)$
for $p<q<p^\ast$. We deduce from $(J|_{S_c})'(u_n)\to0$ that there exists a sequence
$\{\lambda_n\}$ of real numbers such that for $\varphi\in E$,
\begin{align}\label{eq4.6}
\int_{\R^N}|\nabla u_n|^{p-2}\nabla u_n\nabla \varphi
&=\int_{\R^N}|u_n|^{p^\ast-2}u_n\varphi+\mu\int_{\R^N}|u_n|^{q-2}u_n\varphi
\notag\\
&\ \ \ +\lambda_n\int_{\R^N}|u_n|^{p-2}u_n\varphi+o(1)\|\varphi\|
\end{align}
as $n\to\infty$. Taking $u_n$ as a test function in \eqref{eq4.6},
we see that $\{\lambda_n\}$ is bounded. Assume up to a subsequence
that $\lambda_n\to\lambda$ as $n\to\infty$.
If $u=0$, then $u_n\to 0$ in $L^q(\R^N)$, and
\begin{align*}
J(u_n)&=\frac{1}{p} |\nabla u_n|_p^p
-\frac{1}{p^\ast}|u_n|_{p^\ast}^{p^\ast}-\frac{\mu}{q}|u_n|_q^q\notag\\
&\geq\frac{1}{p} |\nabla u_n|_p^p
-\frac{S^{-\frac{p^\ast}{p}}}{p^\ast}|\nabla u_n|_p^{p^\ast}+o(1)\notag\\
&\geq\frac{p^\ast-1}{pp^\ast} |\nabla u_n|_p^p+o(1),
\end{align*}
which is a contradiction with $m<0$. Thus
$u\neq 0$.
A similar argument to the proof of Lemma \ref{lem3.3} shows that
\begin{equation*}\nabla u_n(x)\to \nabla u(x)\ {\text a.e.}\ \in \R^N.\end{equation*}
Thus, taking the limit in \eqref{eq4.6} as $n\to\infty$, we have
\begin{equation}\label{eq4.7}
\int_{\R^N}|\nabla u|^{p-2}\nabla u\nabla \varphi
=\int_{\R^N}|u|^{p^\ast-2}u\varphi+\mu\int_{\R^N}|u|^{q-2}u\varphi
+\lambda\int_{\R^N}|u|^{p-2}u\varphi,
\end{equation}
which means that $u$ is a weak solution of the equation
$$-\Delta_p u=\lambda |u|^{p-2}u+\mu|u|^{q-2}u+|u|^{p^\ast-2}u\ \
\text{in}\ \R^N.$$
We know from \cite[Proposition 2.1]{ll11} that $u$ satisfies the following Poho\v zaev
identity
\begin{equation}\label{eq4.8}
\int_{\R^N}|\nabla u|^p
-\int_{\R^N}|u|^{p^\ast}-\mu \gamma_q \int_{\R^N}|u|^q=0.
\end{equation}
Then, in view of
\eqref{eq4.7}, \eqref{eq4.8} and $\gamma_q<1$, we have
\begin{align*}
\lambda |\nabla u|_p^p=|\nabla u|_p^p
-|u|_{p^\ast}^{p^\ast}-\mu|u|_q^q=\mu(\gamma_q-1)|u|_q^q<0,
\end{align*}
which ensures $\lambda<0$.

Since $u_n\rightharpoonup u$ in $D^{1,p}$,
using the second Concentration Compactness lemma of Lions [15], there exist an at most countable index
set $K$, a family $\{x_i: i\in K\}\subset \mathbb{R}^N$ and two families of positive numbers $\{\mu_i: i\in K\},\ \{\nu_i: i\in K\}$ such that
\begin{equation}\label{eq4.9}
|\nabla u_n|^p\rightharpoonup\mu\geq|\nabla u|^p+\sum_{i\in K}\mu_i\delta_{x_i},\
|u_n|^{p^\ast}\rightharpoonup\nu=|u|^{p^\ast}+\sum_{i\in K}\nu_i\delta_{x_i}
\end{equation}
weakly star convergence in the sense of measures, where $\delta_{x_i}$ is the Dirac mass concentrated at $x_i$
and that
\begin{equation}\label{eq4.10}
\sum_{i\in K}\nu_i^{p/p^\ast}<\infty;\ S\nu_i^{p/p^\ast}\leq\mu_i,\  i\in K.
\end{equation}
Taking $x_i\in\mathbb{R}^N$ in the support of the singular part of $\omega,\nu,\zeta$.  Now for any $\epsilon>0$, we define $\chi_\epsilon(x):=\overline{\chi}_\epsilon(x-x_i)$, where $\overline{\chi}_\epsilon\in C_0^\infty(\mathbb{R}^N,[0,1])$ such that $\overline{\chi}_\epsilon\equiv1$ on $B_\epsilon(0)$, $\overline{\chi}_\epsilon\equiv0$ on $\mathbb{R}^N\backslash B_{2\epsilon}(0)$ and $|\nabla \overline{\chi}_\epsilon|\in [0,\frac{2}{\epsilon}]$.
Now we divide the proof into three steps.

$Step 1:$ For any $i\in K,\  \mu_i\leq\nu_i.$
It is clear that the sequence $\{\chi_\epsilon u_n\}$ is bounded in $E$, then we have
\begin{equation}\label{eq4.11}
\int_{\mathbb{R}^N} u_n|\nabla u_n|^{p-2}\nabla u_n\nabla\chi_\epsilon=\int_{\mathbb{R}^N}\left(-|\nabla u_n|^p+|u_n|^{p^\ast}+\mu|u_n|^{q}+\lambda_n|u_n|^p\right)\chi_\epsilon +o(1).
\end{equation}
Using the H\"{o}lder inequality, we obtain the following limit expression:
\begin{align}\label{eq4.12}
\limsup_{n\to\infty}\left|\int_{\mathbb{R}^N} u_n|\nabla u_n|^{p-2}\nabla u_n\nabla\chi_\epsilon\right|
&\leq\limsup_{n\to\infty}\left(\int_{B_{2\epsilon(x_i)}} |u_n\nabla \chi_\epsilon|^p\right)^{\frac{1}{p}}\left(\int_{B_{2\epsilon(x_i)}} |\nabla u_n|^p\right)^{\frac{p-1}{p}}\notag\\
&\leq C\left(\int_{B_{2\epsilon(x_i)}} |u|^p|\nabla\chi_\epsilon|^p \right)^{\frac{1}{p}}\notag\\
&\leq C\left(\int_{B_{2\epsilon(x_i)}} |u|^{p^\ast}\right)^{\frac{1}{p^\ast}}\left(\int_{B_{2\epsilon(x_i)}} |\nabla\chi_\epsilon|^\frac{pp^\ast}{p^\ast-p}\right)^{\frac{p^\ast-p}{pp^\ast}}\notag\\
&=C\left(\int_{B_{2\epsilon(x_i)}} |u|^{p^\ast}\right)^{\frac{1}{p^\ast}}\to0,
\end{align}
as $\epsilon\to0$. Moreover, since $\chi_\epsilon$ has compact support, there holds
 \begin{equation}\label{eq4.13}\lim_{n\to\infty}\int_{\mathbb{R}^N}|\nabla u_n|^p \chi_\epsilon \geq\int_{\mathbb{R}^N}|\nabla u|^p \chi_\epsilon+\langle\sum_{i\in K}\mu_i\delta_{x_i},\chi_\epsilon\rangle,\end{equation}
\begin{equation}\label{eq4.14}\lim_{n\to\infty}\int_{\mathbb{R}^N}|u_n|^{p^\ast} \chi_\epsilon=\int_{\mathbb{R}^N}|\nabla u|^{p^\ast} \chi_\epsilon+\langle\sum_{i\in K}\nu_i\delta_{x_i},\chi_\epsilon\rangle.\end{equation}
Moreover, if follows from \eqref{eq4.13} and \eqref{eq4.14} that
\begin{equation}\label{eq4.15}
\begin{split}
&\limsup_{n\to\infty}\int_{\mathbb{R}^N} u_n|\nabla u_n|^{p-2}\nabla u_n\nabla\chi_\epsilon
\leq-\int_{\mathbb{R}^N}|\nabla u|^p\chi_\epsilon
+\int_{\mathbb{R}^N}|u|^{p^\ast}\chi_\epsilon\\
&\ \ \
+\int_{\mathbb{R}^N}\mu|u|^{q}\chi_\epsilon+\int_{\mathbb{R}^N}\lambda u^2\chi_\epsilon-\langle\sum_{i\in K}\mu_i\delta_{x_i},\chi_\epsilon\rangle
+\langle\sum_{i\in K}\nu_i\delta_{x_i},\chi_\epsilon\rangle.
\end{split}
\end{equation}
Thus, taking the limit in \eqref{eq4.15} as $\epsilon\to0$ , we have from \eqref{eq4.10} that
\begin{equation}\label{eq4.16}
\mu_i\leq\nu_i.
\end{equation}

$Step 2:$ $\mu_i=0$ for any $i\in K$ and $K=\emptyset$.
Suppose that there exists $i_0\in K$ such that $\mu_{i_0}>0$. Using \eqref{eq4.10} and \eqref{eq4.16}, we obtain
\begin{equation}\label{eq4.17}\mu_{i_0}\geq S^\frac{p^\ast}{p^\ast-p}.\end{equation}
Consequently, by virtue of  \eqref{eq4.13} and \eqref{eq4.17},
$$R_0^p\geq\limsup_{n\to\infty}|\nabla u_n|_p^p\geq\mu_{i_0}
\geq S^\frac{p^\ast}{p^\ast-p},$$
which contradicts with $R_0<t^\ast\leq S^\frac{p^\ast}{p (p^\ast-p)}$. Then $K=\emptyset$, and hence
\begin{equation}\label{eq4.18}u_n\to u\ {\rm in}\ L^{p^\ast}_{loc}(\R^N).\end{equation}

$Step 3:$ $u_n\to u$ in $E$.

Since $\{u_n\}$ is also bounded in $E$, we know from \cite{sww} that

$$|u_n(x)|\leq C|x|^{-\frac{N-1}{p}}\|u_n\|\leq C_1|x|^{-\frac{N-1}{p}},\ {\rm a.e.\ for}\ |x|\geq R,$$
and then
$$|u_n(x)|^{p^\ast}\leq\frac{C_2}{|x|^\frac{N(N-1)}{N-p}},\ {\rm a.e.\ for}\ |x|\geq R.$$
Notice that $\frac{C_1}{|\cdot|^\frac{N(N-1)}{N-p}}\in L^1(\R^N\setminus B_R(0))$ and $u_n(x)\to u(x)$,
a.e. $x\in \R^N\setminus B_R(0)$, then the Lebesgue's Theorem leads to
$$u_n\to u\ {\rm in}\  L^{p^\ast}(\R^N\setminus B_R(0)),$$
which, together with \eqref{eq4.18}, implies that
\begin{equation}\label{eq4.19}u_n\to u\ {\rm in}\  L^{p^\ast}(\R^N).\end{equation}
Thus, taking $\varphi=u_n$ as the test function in \eqref{eq4.6},
\begin{align}\label{eq4.20}
|\nabla u_n|_p^p-\lambda|u_n|_p^p=|u_n|_{p^\ast}^{p^\ast}+\mu|u_n|_q^q
+o(1).
\end{align}
Then, we deduce from \eqref{eq4.18}, \eqref{eq4.19} and \eqref{eq4.20} that
$$\lim_{n\to\infty}(|\nabla u_n|_p^p-\lambda|u_n|_p^p)=|u|_{p^\ast}^{p^\ast}+\mu|u|_q^q.$$
Combining with \eqref{eq4.7} yields
$$\lim_{n\to\infty}(|\nabla u_n|_p^p-\lambda|u_n|_p^p)=|\nabla u|_p^p-\lambda|u|_p^p.$$
Note that $\lambda<0$, we immediately from \eqref{eq2.5} obtain $u_n\to u$ in $E$.
\end{proof}

\begin{proof}[\bf Proof of Theorem \ref{the1.1} (Part I)]
Recalling that $J_T$ is bounded below on $S_c$.
By the definition of $m(c)$, for each $n\in\N_+$ there exists
$v_n\in S_c$ such that $v_n\geq 0$ and
$$J_T(v_n)<m(c)+\frac{1}{n}.$$
Employing the \cite[Theorem 8.5]{w} yields the existence
of $\{u_n\}\subset S_c$ such that
\begin{equation*}
m(c)\leq J_T(u_n)\leq m(c)+\frac{2}{n},\ \ \min_{\lambda\in\R}\|J_T'(u_n)-\lambda |u_n|^{p-2}u_n\|\leq\frac{4}{\sqrt{n}},\
\|u_n-v_n\|\leq\frac{1}{\sqrt{n}}.
\end{equation*}
Then, we obtain that $\{u_n\}$ is a Palais-Smale sequence for $J_T|_{S_c}$
at the level $m(c)$ with $|u_n^-|_p\to0$ as $n\to\infty$.
By Lemma \ref{lem4.5}, we have $u_n\to u$ in $E$ for some $u$, and hence
$J_T(u)=m(c)$. Moreover, we see from Lemmas \ref{lem4.4} and \ref{lem4.5}
that $u$ is in the set
$D_{R_0}(c)$. Therefore, $u$ is also a critical point of $J$ on $S_c$ with $J(u)=m(c)$ and solves \eqref{eq1.1} and \eqref{eq1.3} for
some $\lambda<0$. By $|u_n^-|_p\to 0$ as $n\to\infty$, we see that $u\geq0$.
 Moreover, we borrow the proof in \cite[Theorem 1.1]{sy1}.
It is elementary to realize that $u(x)>0$ for $x\in \R^N$.
The proof is completed.
\end{proof}

\subsection{Infinitely many normalized solutions}

In what follows, in order to formulate the infinitely many solutions, we present a minimax theorem for a class of constrained even functionals that is proved in Jeanjean and Lu \cite{jssl}. We introduce the notation of the genus. Let $\Sigma(S_c)$ be the family of closed symmetric subsets of $S_c$. For any nonempty set
$A\in\Sigma(S_c)$, the genus $\mathcal {G}(A)$ of $A$ is defined as the least integer
$k\geq1$ for which there exists an odd continuous mapping $\varphi:A\to\R^k\setminus\{0\}$.
We set $\mathcal {G}(A)=\infty$ if such an integer does not exist, and set
$\mathcal {G}(A)=0$ if $A=\emptyset$. For each $k\in\N_+$, let
$\Gamma_k:=\{A\in\Sigma(S_c)|\mathcal {G}(A)\geq k\}$.
We shall need some basic properties of the genus. For $A \subset S_c$ and $\delta>0$, denote by $A_{\delta}$ the uniform $\delta$-neighborhood of $A$ in $S_c$, that is,
$$A_\delta:=\{u\in E|\ \inf_{v\in A}\|u-v\|<\delta\}.$$
Since $S_c$ is a closed symmetric subset of $E$, repeating the arguments in \cite[Section 7]{phr}, we have the following lemma.

\begin{lemma}\cite[Proposition 2.2]{jssl}\label{lem4.6} Let $A, B \in \Sigma(S_c)$. Then the following statements hold.\\
(i) If $\mathcal{G}(A) \geq 2$, then $A$ contains infinitely many distinct points.\\
(ii) $\mathcal{G}(\overline{A \backslash B}) \geq \mathcal{G}(A)-\mathcal{G}(B)$ if $\mathcal{G}(B)<\infty$.\\
(iii) If there exists an odd continuous mapping $\psi: \mathbb{S}^{k-1} \rightarrow A$, then $\mathcal{G}(A) \geq k$.\\
(iv) If $A$ is compact, then $\mathcal{G}(A)<\infty$ and there exists $\delta>0$ such that $A_{\delta} \in \Sigma(S_c)$ and $\mathcal{G}(A_\delta)=\mathcal{G}(A)$.
\end{lemma}
We shall also need the following quantitative deformation lemma given in \cite[Theorem 3.11]{ms}. For $m\in \R$, set $I|_{S_c}^{m}:=\{u \in S_c \mid I(u) \leq m\}$.

\begin{lemma}\label{lem4.7} Suppose $I|_{S_c} \in C^{1}(S_c)$ satisfies (PS) condition. Let $m \in \mathbb{R}, \bar{\varepsilon}>0$ be given and let $N$ be any neighborhood of $K^{m}$. Then there exist a number $\varepsilon \in( 0, \bar{\varepsilon})$ and a continuous 1-parameter family of homeomorphisms $\eta(t,\cdot)$ of $S_c, 0 \leq t<\infty$, with the properties\\
(i) $\eta(t, u)=u$, if $t=0$, or $I|_{S_c}'(u)=0$, or $|I|_{S_c}(u)-m| \geq \bar{\varepsilon}$;\\
(ii) $I|_{S_c}(\eta(t, u))$ is non-increasing in $t$ for any $u \in S_c$;\\
(iii) $\eta\left(1 ,I|_{S_c}^{m+\varepsilon} \backslash N\right) \subset I|_{S_c}^{m-\varepsilon}$, and
$\eta\left(1, I|_{S_c}^{m+\varepsilon}\right) \subset I|_{S_c}^{m-\varepsilon} \cup N$;\\
(iv) $\eta(t, u)$ is odd in $u \in S_c$ for any $t \in[0,1]$ if $I|_{S_c}$ is even.
\end{lemma}

With the help of Proposition $2.2$ and Lemma $2.3$ and arguing as in the beginning of the proof
of \cite[Theorem 2.1]{jssl}, we have the following Lemma.

\begin{lemma}\label{lem4.8}Let $I: E \rightarrow \mathbb{R}$ be an even functional of class $C^{1}$. Assume that $I|_{S_c}$ is bounded from below and satisfies the $(P S)_{m}$ condition for all $m<0$, and that $\Gamma_{k} \neq \emptyset$ for each $k \in \mathbb{N}$. Then a sequence of minimax values $-\infty<m_{1} \leq m_{2} \leq \cdots \leq m_{k} \leq \cdots$ can be defined as follows:
$$m_{k}:=\inf _{A \in \Gamma_{k}} \sup _{u \in A} I(u), \  k \geq 1,$$
and the following statements hold.\\
(i) $m_{k}$ is a critical value of $I|_{S_c}$ provided $m_{k}<0$.\\
(ii) Denote by $K^{m}$ the set of critical points of $I|_{S_c}$ at a level $m \in \mathbb{R}$. If
$$m_{k}=m_{k+1}=\cdots=m_{k+l-1}=: m<0 \quad \text { for some } k, l \geq 1,$$
then $\mathcal{G}\left(K^{m}\right) \geq l$. In particular, $I|_{S_c}$ has infinitely many critical points at the level $m$ if $l \geq 2$.\\
(iii) If $m_{k}<0$ for all $k \geq 1$, then $m_{k} \to 0^{-}$as $k \to \infty$.
\end{lemma}
\begin{proof}[\bf Proof]  Observing that Item (i) is a special case of Item (ii) when $l=1$, it is enough to prove Item (ii).
 We show by contradiction and assume that $\mathcal{G}(K^{m}) \leq l-1$.
It is easy to see that $K^{m} \in \Sigma(S_c)$ and $K^{m}$ is compact by virtue of the $(P S)_{m}$ condition. Thus  Lemma \ref{lem4.6} (iv) implies that there exists $\delta>0$ such that
$K_{3 \delta}^{m} \subset S_c$ and
$$\mathcal{G}(\overline{K_{3 \delta}^{m}})=\mathcal{G}(K^{m}) \leq l-1.$$
We remark here that $\overline{K_{3 \delta}^{m}}=\emptyset$ if $K^{m}=\emptyset$. If we set $N=K_{3 \delta}^{m}$, then, by using Lemma \ref{lem4.7}, there exist a number $\varepsilon>0$ and a mapping $\eta \in C([0,\infty)\times S_c, S_c)$ such that
$$
\eta\left(1, I|_{S_c}^{m+\varepsilon}\backslash N\right) \subset I|_{S_c}^{m-\varepsilon} \quad \text { and } \quad
\eta(t,\cdot)\ \text { is odd for all }\ t\in[0,\infty).
$$
Choose $A \in \Gamma_{k+l-1}$ such that $\sup_{u \in A} I(u) \leq m+\varepsilon$. It is clear that $A \backslash K_{3\delta}^{m} \subset I|_{S_c}^{m+\varepsilon}\backslash K_{3\delta}^{m}$ and hence
$$
\eta(1, A \backslash K_{3\delta}^{m}) \subset \eta(1, I|_{S_c}^{m+\varepsilon}\backslash K_{3\delta}^{m})
\subset I|_{S_c}^{m-\varepsilon}.
$$
On the other hand, since $\mathcal{G}(A \backslash K_{3 \delta}^{m}) \geq \mathcal{G}(A)-\mathcal{G}(K_{3\delta}^{m}) \geq k$ by Lemma \ref{lem4.6} (ii), we derive $A \backslash K_{3\delta}^{m} \in \Gamma_{k}$ and then $\eta(1, A \backslash K_{3 \delta}^{m}) \in \Gamma_{k}$. Now, by the definition of $m_{k}$, we get a contradiction:
$$
m=m_{k} \leq \sup _{u \in \eta(1, A \backslash K_{3\delta}^{m})} I(u) \leq m-\varepsilon.
$$
Thus $\mathcal{G}\left(K^{m}\right) \geq l$. In view of Lemma \ref{lem4.6} (i), we complete the proof of Item (ii).

To prove Item (iii), we assume by contradiction that there exists $m<0$ such that $m_{k} \leq m$ for all $k \geq 1$ and $m_{k} \to m$ as $k \to \infty$. Making use the $(P S)_{m}$ condition again, $K^{m}$ is a (symmetric) compact set. Thus, Lemma \ref{lem4.6} (iv) leads to there exists $\delta>0$ such that
$$\mathcal{G}(\overline{K_{3 \delta}^{m}})=\mathcal{G}(K^{m})=: q<\infty.$$
Let $N:=K_{3\delta}^{m} \subset S_c$. Then we know from Lemma \ref{lem4.7} that there exist $\varepsilon>0$ and a mapping $\eta \in C([0,\infty) \times S_c, S_c)$ such that $\eta(1, I|_{S_c}^{m+\varepsilon}\backslash K_{3\delta}^{m}) \subset I|_{S_c}^{m-\varepsilon}$ and $\eta(t, \cdot)$ is odd for any $t \in[0,1]$. Choose $k \geq 1$ large enough such that $m_k>m-\varepsilon$ and take $A \in \Gamma_{k+q}$ such that $\sup_{u\in A}I(u) \leq m_{k+q}+\varepsilon$. Noting that $m_{k+q} \leq m$, we have $A\backslash K_{3\delta}^{m} \subset I|_{S_c}^{m+\varepsilon} \backslash K_{3\delta}^{m}$ and thus
$$
\eta(1, A \backslash K_{3\delta}^{m}) \subset \eta(1, I|_{S_c}^{m+\varepsilon} \backslash K_{3\delta}^{m}) \subset I|_{S_c}^{m-\varepsilon}.
$$
On the other hand, since $\mathcal{G}(A \backslash K_{3\delta}^{m}) \geq \mathcal{G}(A)-\mathcal{G}(K_{3 \delta}^{m}) \geq k$, we have $A \backslash K_{3\delta}^{m} \in \Gamma_{k}$ and then $\eta(1, A \backslash K_{3\delta}^{m}) \in \Gamma_{k}$. Hence,
$$
m_{k} \leq \sup _{u \in \eta(1, A \backslash K_{3\delta}^{m})} I(u) \leq m-\varepsilon,
$$
for we chosen $k$ large enough such that $m_{k}>m-\varepsilon$, which is impossible. Thus $c_{k} \to 0^{-}$as $k \to \infty$.
\end{proof}

Now, for $\varepsilon>0$, let us introduce the set
$$B_\varepsilon=\{u\in S_c:\ J_T(u)\leq-\varepsilon\}\subset E,$$
which is a closed symmetric subset of $S_c$, because $J_T$ is even and continuous.

\begin{lemma}\label{lem4.9}
Let $n\in \N^\ast$, there exists $\varepsilon(n)>0$ such that $\mathcal {G}(B_\varepsilon)\geq n$
for any $0<\varepsilon\leq\varepsilon(n)$.
\end{lemma}

\begin{proof}[\bf Proof]
For each $n\in\N^\ast$,
we can find $n$ functions $u_1,u_2,\ldots, u_n\in C_0^\infty(\R^N)$ of linearly independent with the ${\rm supp}u_i\cap {\rm supp} u_j=\emptyset$ for $i\neq j$. The $n$-dimensional subspace $E_n$ is defined by $E_n = {\rm span}\{u_1,u_2,\ldots, u_n\}$ equipped with the norm of $E$.
Moreover, there holds $|\nabla u_j|_p=\rho<R_0$
and $|u_j|_p=c$, and hence $\|u_j\|=(\rho^p+c^p)^{1/p}$.  Furthermore, for $t\in\R$, we introduce the set
$$\Upsilon_n(t)=\{s_1u_1^t+s_2u_2^t+\cdots+s_nu_n^t:\ s_1^p+s_2^p+\cdots+s_n^p=1\}.$$
It is easy to verify that there exists a homomorphism between $\Upsilon_n(t)$
and the sphere $\mathcal {D}(t)=\{(y_1,y_2,\ldots,y_n)\in\R^n:\ y_1^p+y_2^p+\cdots+y_n^p=
e^{pt}\rho^p+c^p\}$ in $\R^n$ for $t\in\R$.
As a consequence, by the properties of genus, one has $\mathcal {G}(\Upsilon_n(t))=n$.
Let $s_1^p+s_2^p+\cdots+s_n^p=1$, $v=s_1u_1^t+s_2u_2^t+\cdots+s_nu_n^t\in \Upsilon_n(t)$ with $t<0$, there holds $|\nabla v|_p=e^t\rho<R_0$.
At this point, we observe that
$$J_T(v)=J(v)=\frac{e^{pt}}{p} \rho^p
-\frac{e^{p^\ast t}}{p^\ast}\rho^{p^\ast}\int_{\R^N}\left|\frac{u}{\rho}\right|^{p^\ast}
-\frac{\mu e^{q\gamma_qt}}{q}\rho^q\int_{\R^N}\left|\frac{u}{\rho}\right|^q,$$
where $u=s_1u_1+s_2u_2+\cdots+s_nu_n$.
On the other hand, let us define
$$\alpha_n=\inf\{|w|_{p^\ast}^{p^\ast}:\ w\in E_n,\ |\nabla w|_p=1\}>0$$
and
$$\beta_n=\inf\{|w|_q^q:\ w\in E_n,\ |\nabla w|_p=1\}>0.$$
Thus, we have
$$J_T(v)\leq\frac{e^{pt}}{p} \rho^p
-\frac{e^{{p^\ast}t}}{p^\ast}\rho^{p^\ast}\alpha_n
-\frac{\mu e^{q\gamma_qt}}{q}\rho^q\beta_n.$$
Note that $0<q\gamma_q<p$, there exist $\varepsilon(n)>0,\ t_n<0$ such that
for $0<\varepsilon\leq\varepsilon(n)$,
$$J_T(v)\leq -\varepsilon,\ \ {\rm for\ all }\ v\in \Upsilon_n(t_n).$$
Hence, $\Upsilon_n(t_n)\subset B_\va$ and
$\mathcal {G}(B_\va)\geq\mathcal {G}(\Upsilon_n(t_n))=n$.
\end{proof}
\begin{proof}[{\bf Proof of Theorem \ref{the1.1}\ (Part II)}]

We set $\Gamma_k=\{A\in \Sigma(S_c) : \mathcal {G}(A)\geq k\}$
for $k\geq 1$. By Lemma \ref{lem4.9}, there exists $\va(k)>0$
such that $B_{\va(k)}\in\Gamma_k$, where
$B_{\va(k)}=\{u\in S_c : J_T(u)\leq -\va(k)\}$. Then
$\Gamma_k\neq \emptyset$. Recall that $J_T$ is bounded
from below and, by Lemma \ref{lem4.5}, $J_T$ satisfies the
$(PS)_m$ condition on $S_c$ at the level $m<0$. We define
the minimax values
$$
m_k=\inf_{A\in\Gamma_k}\sup_{u\in A}J_T(u),\ \ \text{for}\ k\geq 1.
$$
Then $-\infty<m_1\leq m_2\leq\cdots\leq m_k\leq\cdots<0$.
By Lemma \ref{lem4.8}, $J_T|_{S_c}$ has a critical point at
the level $m_k$ and $m_k\to0^-$ as $k\to\infty$. Since
$J_T=J$ in a small neighborhood of $u$ provided that
$J_T(u)<0$, these critical points of $J_T|_{S_c}$ are indeed
critical points of $J|_{S_c}$. The proof is complete.

\end{proof}

\subsection{Another positive normalized
solution of mountain pass type}

Recall that
$$\varphi_c(t)=\frac{1}{p}t^{p-\gamma_q q}
-\frac{S^{-\frac{p^\ast}{p}}}{p^\ast}t^{p^\ast-\gamma_q q}
-\frac{\mu}{q}C(q)c^{(1-\gamma_q)q}.$$

\begin{lemma}\label{lem4.10}
Let $c_1>0$ and $t_1>0$ be such that $\varphi_{c_1}(t_1)\geq 0$.
If $c_2\in(0,c_1]$, then
$$\varphi_{c_2}(t)\geq 0,\ \ \text{for}\ t\in\Big[\frac{c_2}{c_1}t_1,t_1\Big].$$
\end{lemma}

\begin{proof}[\bf Proof]
It is easy to verify that
$\varphi_{c_2}(t_1)\geq \varphi_{c_1}(t_1)\geq0$ and
$$\varphi_{c_2}\left(\frac{c_2}{c_1}t_1\right)\geq \left(\frac{c_2}{c_1}\right)^{p-\gamma_qq}\varphi_{c_1}(t_1)\geq0.$$
Suppose that the conclusion of the Lemma is not satisfied. We may assume $\varphi_{c_2}(t)<0$ for some $t\in \big(\frac{c_2}{c_1}t_1,t_1\big)$,
which therefore implies $\varphi_{c_2}$ has a local minimum point in
$\big(\frac{c_2}{c_1}t_1,t_1\big)$. This is a contradiction,
since the function $\varphi_{c_2}$ has a unique critical point
at which $\varphi_{c_2}$ achieves its global maximum. The proof
is completed.
\end{proof}

For $\mu>0$, let us define
$$c^\ast:=c^\ast(\mu)=\left(\frac{\alpha_q}{\mu}\right)^{\frac{1}{(1-\gamma_q)q}}.$$

\begin{lemma}\label{lem4.11}
$(i)$ The map $c\in(0,c^\ast)\mapsto m(c)\in\R$ is continuous.\\
$(ii)$ If $c_0\in(0,c^\ast)$, then
$$m(c_0)\leq m(c)+m\left(\sqrt[p]{c_0^p-c^p}\right),\ \ \text{for}\ c\in(0,c_0).$$
Moreover, if $m(c)$ or $m\big(\sqrt[p]{c_0^p-c^p}\big)$ is achieved, then the above
inequality is strict.
\end{lemma}

\begin{proof}[\bf Proof]
$(i)$ Let $\{c_n\}\subset(0,c^\ast)$ be such that
$\lim_{n\to\infty}c_n=c\in(0,c^\ast)$. For $\va>0$ small,
there exists $u\in A_{t^\ast}(c)$ such that
$J(u)\leq m(c)+\va<m(c)+3\va<0$. Setting $u_n=c_n u/c$, we
have $|u_n|_p=c_n$ and $|\nabla u_n|_p<t^\ast$ for $n$ large,
which means that $u_n\in A_{t^\ast}(c_n)$ for $n$ large. Then,
since $J(u_n)\to J(u)$ as $n\to\infty$, we have
\begin{equation}\label{eq4.21}
m(c_n)\leq J(u_n)\leq J(u)+\va\leq m(c)+2\va
\end{equation}
for $n$ large.

For above $\va>0$, we see from \eqref{eq4.21} that there
exists $u_n\in A_{t^\ast}(c_n)$ such that
$$J(u_n)\leq m(c_n)+\va\leq m(c)+3\va<0$$
for $n$ large. Moreover, setting $v_n=cu_n/c_n$, it follows immediately that $|v_n|_2=c$.
We wish to obtain
\begin{equation}\label{eq4.22}
|\nabla v_n|_2<t^\ast
\end{equation}
and hence $v_n\in A_{t^\ast}(c)$. If this is valid, there holds
\begin{equation}\label{eq4.23}
m(c)\leq J(v_n)\leq J(u_n)+\va\leq m(c_n)+2\va
\end{equation}
for $n$ large. Combining \eqref{eq4.21} and \eqref{eq4.23},
we conclude that $m(c_n)\to m(c)$ as $n\to\infty$.

It remains to prove \eqref{eq4.22}. In the case $c_n\geq c$,
we can infer that
$$|\nabla v_n|_2=\frac{c}{c_n}|\nabla u_n|_2\leq|\nabla u_n|_2<t^\ast.$$
When $c_n< c$, it follows from Lemma \ref{lem4.10} and
$\varphi_c(t^\ast)>0$ that
$$\varphi_{c_n}(t)\geq0,\ \ \text{for}\ t\in\Big[\frac{c_n}{c}t^\ast,t^\ast\Big].$$
Consequently, on the basis of $J(u_n)<0$ and $u_n\in A_{t^\ast}(c_n)$, we
have $|\nabla u_n|_2<\frac{c_n}{c}t^\ast$ and henceforth
$$|\nabla v_n|_2=\frac{c}{c_n}|\nabla u_n|_2<t^\ast.$$

$(ii)$ Letting $c_0\in(0,c^\ast)$ and $c\in(0, c_0)$, we start showing that
\begin{equation}\label{eq4.24}
m(\theta c)\leq \theta^2 m(c),\ \ {\rm for}\ \theta\in \Big(1,\frac{c_0}{c}\Big],
\end{equation}
and if in addition $m(c)$ is achieved then the inequality is
strict. Indeed, for $\va>0$ small there exists $u\in A_{t^\ast}(c)$
such that
\begin{equation}\label{eq4.25}
J(u)\leq m(c)+\va<0.
\end{equation}
Exploiting the fact that $\varphi_{c_0}(t^\ast)\geq0$, we see from Lemma \ref{lem4.10}
that $\varphi_{c}(t)\geq0$ for $t\in[\frac{c}{c_0}t^\ast,t^\ast]$.
Since $J(u)<0$ and $u\in A_{t^\ast}(c)$, there must be
$|\nabla u|_p<\frac{c}{c_0}t^\ast$.
Furthermore, setting $v(x)=\theta u(x)$,
it is now easy to observe
that $|v|_p=\theta|u|_p=\theta c$
and $|\nabla v|_p=\theta|\nabla u|_p<t^\ast$, which means
that $v\in A_{t^\ast}(\theta c)$. Then, in view of \eqref{eq4.25},
\begin{align*}
m(\theta c)\leq J(v)=\frac{\theta^p}{p}|\nabla u|_p^p
-\frac{\theta^{p^\ast}}{{p^\ast}}|u|_{p^\ast}^{p^\ast}-\frac{\mu}{q}\theta^q|u|_q^q
< \theta^p J(u)\leq \theta^p(m(c)+\va).
\end{align*}
Note that $\va>0$ is arbitrary, it is clear $m(\theta c)\leq\theta^p m(c)$.
Let us emphasize that here, if $m(c)$ is achieved, then one can choose $\va=0$ in \eqref{eq4.25}
and henceforth the strict inequality holds.

Finally, applying \eqref{eq4.24}, we deduce that
$$m(c_0)
=\frac{c^p}{c_0^p}m\left(\frac{c_0}{c}c\right)
+\frac{c_0^p-c^p}{c_0^p}m\left(\frac{c_0}{\sqrt[p]{c_0^p-c^p}}\sqrt[p]{c_0^p-c^p}\right)
\leq m(c)+m\left(\sqrt[p]{c_0^p-c^p}\right)$$
with a strict inequality if $m(c)$ or $m\left(\sqrt[p]{c_0^p-c^p}\right)$ is achieved.
\end{proof}

\begin{remark}\label{rem4.1}
As a consequence of Lemma \ref{lem4.11}\,$(ii)$,
$c\in(0,c^\ast)\mapsto m(c)\in\R$ is decreasing.
\end{remark}

In the following, we are devoted to the existence of second
critical point for $J|_{S_c}$.

\begin{lemma}\label{lem4.12}
For $u\in S_c$ with $J(u)<m(c)$, the value $t_u$ given
in Lemma \ref{lem4.3} is negative.
\end{lemma}

\begin{proof}[\bf Proof]
Let $u\in S_c$ be such that $J(u)<m(c)$ and $s_u<c_u<t_u<d_u$
be as in Lemma \ref{lem4.3}. We shall show $d_u\leq 0$ and
thus $t_u<0$. Let us assume now, by contradiction, that $d_u>0$. It
follows from $\Psi_u(0)=J(u)<m(c)<0$ that $c_u>0$. Then making use of
Lemma \ref{lem4.3} leads to
\begin{align*}
m(c)
&>J(u)=\Psi_u(0)\geq\min_{s\in(-\infty, c_u)}\Psi_u(s)\\
&\geq\min\left\{J(u^s) : s\in \R,\ |\nabla u^s|_p<R_0\right\}
=J(u^{s_u})\geq m(c),
\end{align*}
which is absurd. The proof is completed.
\end{proof}

The first part of Theorem \ref{the1.1}  states that
there exists $u_\ast\in \cM^+$ such that $J(u_\ast)=m(c)<0$.
It is also clear that $J\big(u_\ast^t\big)<2m(c)$ for $t>0$ sufficiently
large. This implies, by Lemma \ref{lem2.3}, that the set
\begin{equation*}
\Gamma:=\Gamma(c,\mu)=\left\{\gamma\in C([0,1],S_c):
\gamma(0)\in \cM^+,\ \gamma(1)\in J^{2m(c)}\right\}
\end{equation*}
is not empty, where $J^{m(c)}=\{u\in S_c: J(u)\leq 2m(c)\}$. Now let us consider the associated minimax value
$$m^\ast:=m^\ast(c,\mu)
=\inf_{\gamma\in\Gamma}\sup_{t\in[0,1]}J(\gamma(t)).$$

\begin{lemma}\label{lem4.13}
There holds $m^\ast=\inf_{u\in \cM^-}J(u)>0$.
\end{lemma}

\begin{proof}[\bf Proof]
A similar argument to the proof of \cite[Lemma 5.7]{ns2} shows that $\inf_{u\in \cM^-}J(u)>0$.
For any $\gamma\in \Gamma$, we have $0=s_{\gamma(0)}<t_{\gamma(0)}$
and, due to Lemma \ref{lem4.12}, $t_{\gamma(1)}<0$. With the aid of Lemma \ref{lem4.3},
there exists $t_0=t_0(\gamma)\in(0,1)$ such that $t_{\gamma(t_0)}=0$, which implies
$\gamma(t_0)\in \mathcal {M}^-$ and necessarily
\begin{equation}\label{eq4.26}\sup_{t\in[0,1]}J(\gamma(t))\geq J(\gamma(t_0))
\geq\inf_{u\in \cM^-}J(u).\end{equation}
Clearly, $m^\ast\geq\inf_{u\in \cM^-}J(u)$. For any $u\in \cM^-$, we
know $s_u<0$ and there exists $t_1=t_1(u)>0$ such that
$J(u^{t_1})<2m(c)$. If we define $\gamma: [0,1]\to S_c$ by $\gamma(t)=u^{(1-t)s_u+tt_1}$,
then $\gamma\in \Gamma$ from Lemma \ref{lem2.3}, and by \eqref{eq4.26},
$$J(u)=\sup_{t\in[0,1]}J(\gamma(t))\geq m^\ast.$$
Hence, $\inf_{u\in \cM^-}J(u)\geq m^\ast$, concluding the proof.
\end{proof}

\begin{lemma}\label{lem4.14}
Assume that $N>2$ and $N^{2/3}<p<3$. Then
$m^\ast<m(c)+
\frac{1}{N}S^{N/p}$.
\end{lemma}

\begin{proof}[\bf Proof]
Let $v_{\va,\tau}=u_\ast+\tau u_\va$, where $u_\va$ is given in
Section 2 and $\tau\geq 0$. Taking
into account that
$$w_{\va,\tau}(x)=s^{\frac{N-p}{p}}v_{\va,\tau}(sx),$$
we obtain that, by direct computations,
\begin{equation}\label{eq4.27}
|\nabla w_{\va,\tau}|_p^p=|\nabla v_{\va,\tau}|_p^p,\ \
|w_{\va,\tau}|_{p^\ast}^{p^\ast}=|v_{\va,\tau}|_{p^\ast}^{p^\ast},
\end{equation}
and
\begin{equation}\label{eq4.28}
|w_{\va,\tau}|_p^p=s^{-p}|v_{\va,\tau}|_p^p,\ \ |w_{\va,\tau}|_q^q=s^{\gamma_qq-q}|v_{\va,\tau}|_q^q.
\end{equation}
Choosing $s=\frac{|v_{\va,\tau}|_p}{c}$, we have $w_{\va,\tau}\in S_c$
and then according to Lemma \ref{lem4.3} there exists $t_{\va,\tau}$ such that $w_{\va,\tau}^{t_{\va,\tau}}\in \cM^-$. Recording that $u_\ast\in \cM^+$,
we have $t_{\va,0}>0$. Gathering \eqref{eq2.2} and
$$e^{pt_{\va,\tau}}|\nabla w_{\va,\tau}|_p^p
=e^{p^\ast t_{\va,\tau}}|w_{\va,\tau}|_{p^\ast}^{p^\ast}
+\mu \gamma_qe^{\gamma_qq t_{\va,\tau}}|w_{\va,\tau}|_q^q,$$
we obtain $t_{\va,\tau}\to -\infty$ as $\tau\to +\infty$ uniformly
for $\va>0$ small. Note that $t_{\va,\tau}$ is continuous with
respect to $\tau$. Then for $\va>0$ small there exists $\tau_\va>0$
such that $t_{\va,\tau_\va}=0$, which means $w_{\va,\tau_\va}\in\cM^-$.
Hence, necessarily in view of Lemma \ref{lem4.13},
\begin{equation}\label{eq4.29}
m^\ast\leq J(w_{\va,\tau_\va})\leq\sup_{\tau\geq 0}J(w_{\va,\tau}).
\end{equation}

In view of \eqref{eq4.27} and \eqref{eq4.28},
we have
\begin{align}\label{eq4.30}
J(w_{\varepsilon,\tau})
&=\frac{1}{p} |\nabla v_{\va,\tau}|_p^p
-\frac{1}{p^\ast}|v_{\va,\tau}|_{p^\ast}^{p^\ast}
-\frac{\mu}{q}s^{\gamma_qq-q}|v_{\va,\tau}|_q^q\notag\\
&=\frac{1}{p}|\nabla (u_\ast+\tau u_\va)|_p^p-\frac{1}{p^\ast}|u_\ast+\tau u_\va|_{p^\ast}^{p^\ast}\notag\\
&\ \ \ -\frac{\mu}{q}s^{\gamma_qq-q}|u_\ast+\tau u_\va|_q^q.
\end{align}
It is easy to check that for $p\in(2,3)$, there exists $C_p>0$ such that
\begin{equation}\label{eq4.29*}(1+t)^p\leq 1+t^p+pt+C_pt^2,\ t\geq0.\end{equation}
At this point observing that
$$s^p=\frac{|v_{\va,\tau}|_p^p}{c^p}\leq1+\frac{p\tau}{c^p}\int_{\R^N}u^{p-1}_\ast u_\va+\frac{C_p\tau^{2}}{c^p}\int_{\R^N}u_\ast^{p-2} u^2_\va+\frac{\tau^p}{c^p}|u_\va|_p^p.$$
Then, it is easy to check that
$$\lim_{\tau\to0^+}J(w_{\varepsilon,\tau})\leq J(u_\ast)=m(c),\ \ \text{uniformly
for}\ \varepsilon>0\ \text{small},$$
and thus this means that there exists $\tau_0>0$ such that
\begin{equation}\label{eq4.31}
J(w_{\varepsilon,\tau})<m(c)+\frac{1}{N}S^{N/p},\ \ \text{for}\ \varepsilon>0\
\text{small and}\ \tau\in (0, \tau_0).
\end{equation}
Next, we compute by employing \eqref{eq4.30} and \eqref{eq4.29*},
\begin{align*}
J(w_{\va,\tau})
&\leq J(u_\ast)+\frac{\tau^p}{p} |\nabla u_\va|_p^p
-\frac{\tau^{p^\ast}}{p^\ast}|u_\va|_{p^\ast}^{p^\ast}+\frac{\tau^2C_p}{p}\int_{\R^N}|\nabla u_\ast|^{p-2}|\nabla u_\va|^2\\
&\ \ \
+\tau\int_{\R^N}|\nabla u_\ast|^{p-1}|\nabla u_\va|-\frac{1}{{p^\ast}}\int_{\R^N}\left(|u_\ast+\tau u_\va|^{p^\ast}
-u_\ast^{p^\ast}-\tau^{p^\ast} u_\va^{p^\ast}\right)
+\frac{\mu}{q}|u_\ast|_q^q\\
&\leq J(u_\ast)+\frac{\tau^p}{p} |\nabla u_\va|_p^p
-\frac{\tau^{p^\ast}}{p^\ast}|u_\va|_{p^\ast}^{p^\ast}+\tau\int_{\R^N}|\nabla u_\ast|^{p-2}\nabla u_\ast\nabla u_\va+\frac{\mu}{q}|u_\ast|_q^q.
\end{align*}
And consequently, there exists $\tau_1>0$ such that
\begin{equation}\label{eq4.32}
J(w_{\va,\tau})<m(c)+\frac{1}{N}S^{N/p},\ \ \text{for}\ \varepsilon>0\
\text{small and}\ \tau\in (\tau_1,+\infty).
\end{equation}
Finally, let us now deal with the case $\tau_0\leq \tau\leq\tau_1$. And since $u_\ast$ is a positive solution of  \eqref{eq1.6}
for some $\lambda^\ast<0$, we have
 \begin{align*}
\int_{\R^N}|\nabla u_\ast|^{p-2}\nabla u_\ast\nabla u_\va
=\int_{\R^N}u_\ast^{p^\ast-1} u_\va
+\mu\int_{\R^N}u_\ast^{q-1} u_\va
+\lambda^\ast \int_{\R^N}u_\ast^{p-1} u_\va,
\end{align*}
and
$$\lambda^\ast c^p=\lambda^\ast |u_\ast|_p^p
=\mu(\gamma_q-1)|u_\ast|_q^q.$$
In view of \eqref{eq4.27} and \eqref{eq4.28},
we use the approach in Garcia Azorero and Peral Alonso\cite{jpa1}, where the
following estimate is obtained (see pages 946 and 949):
\begin{align*}
J(w_{\varepsilon,\tau})
&=\frac{1}{p} |\nabla v_{\va,\tau}|_p^p
-\frac{1}{p^\ast}|v_{\va,\tau}|_{p^\ast}^{p^\ast}
-\frac{\mu}{q}s^{\gamma_qq-q}|v_{\va,\tau}|_q^q\notag\\
&\leq\frac{1}{p}|\nabla u_\ast|_p^p+\tau\int_{\R^N}|\nabla u_\ast|^{p-2}\nabla u_\ast\nabla u_\va+\frac{\tau^p}{p}|\nabla u_\va|_p^p-\frac{1}{p^\ast}|u_\ast+\tau u_\va|_{p^\ast}^{p^\ast}\notag\\
&\ \ \ -\frac{\mu}{q}s^{\gamma_qq-q}|u_\ast+\tau u_\va|_q^q+O(\varepsilon^{\frac{N-p}{p}})
\end{align*}
 for $\varepsilon$ small and $\tau_0\leq \tau\leq\tau_1$.

Then, we deduce Lemma \ref{lem2.1}, \eqref{eq2.2}$-$\eqref{eq2.4} that
\begin{align}\label{eq4.33}
&J(w_{\va,\tau})\\
&\leq J(u_\ast)+\frac{\tau^p}{p} |\nabla u_\va|_p^p
-\frac{\tau^{p^\ast}}{p^\ast}|u_\va|_{p^\ast}^{p^\ast}+O(\varepsilon^{\frac{N-p}{p}})\notag\\
&\ \ \ -\frac{1}{p^\ast}\int_{\R^N}\left(|u_\ast+\tau u_\va|^{p^\ast}
-u_\ast^{p^\ast}-(\tau u_\va)^{p^\ast}-{p^\ast}\tau u_\ast^{p^\ast}u_\va\right)\notag\\
&\ \ \ -\frac{\mu}{q}\int_{\R^N}\left(|u_\ast
+\tau u_\va|^q-u_\ast^q-q\tau u_\ast^{q-1}u_\va\right)
+\tau\lambda^\ast\int_{\R^N}u_\ast^{p-1} u_\va\notag\\
&\ \ \ -\frac{\tau\mu}{c^p}(\gamma_q-1)|u_\ast+\tau u_\va|_q^q\int_{\R^N}u_\ast^{p-1} u_\va
-\frac{\tau^p\mu}{pc^p}(\gamma_q-1)|u_\ast+\tau u_\va|_q^q|u_\va|_p^p\notag\\
&\ \ \
 -\frac{C_p\tau^{p-1}\mu}{c^p}(\gamma_q-1)|u_\ast+\tau u_\va|_q^q\int_{\R^N}u_\ast^{p-2} u_\va^2
\notag\\
&\leq J(u_\ast)+\frac{1}{N}S^{N/p}+O(\varepsilon^{\frac{N-p}{p}})-\frac{\tau^p\mu}{pc^p}(\gamma_q-1)|u_\ast+\tau u_\va|_q^q|u_\va|_p^p-\int_{\R^N}u_\ast(\tau u_\va)^{p^\ast-1}\notag\\
&\ \ \ -\frac{\tau\mu}{c^p}(\gamma_q-1)\left(|u_\ast+\tau u_\va|_q^q-|u_\ast|_q^q\right)
\int_{\R^N}u_\ast^{p-1}  u_\va
\notag\\
&\ \ \ -\frac{C_p\tau^{2}\mu}{c^p}(\gamma_q-1)|u_\ast+\tau u_\va|_q^q\int_{\R^N}u_\ast^{p-2} u_\va^2+O(\varepsilon^{\frac{N-p}{p}})\notag\\
&\leq m(c)+\frac{1}{N}S^{N/p}+O(\varepsilon^{\frac{N-p}{p}})+O(\varepsilon^{\frac{2(N-p)}{p^2}})-C \varepsilon^{\frac{(N-p)(p-1)}{p^2}}\notag\\
&<m(c)+\frac{1}{N}S^{N/p}
\end{align}
for $\va>0$ small, where we used $N^{2/3}<p<3$, $p<q<p+p^2/N$ and \eqref{eq2.4}. Therefore,
summarizing \eqref{eq4.29}, \eqref{eq4.31}, \eqref{eq4.32}
and \eqref{eq4.33} yields
$m^\ast\leq\sup_{\tau\geq 0}J(w_{\va,\tau})<m(c)+\frac{1}{N}S^{N/p}$.
We conclude the proof.
\end{proof}

\begin{lemma}\label{lem4.15}
Suppose that $p< q<p+p^2/N$ and $\mu c^{(1-\gamma_q)q}<\alpha_q$. Then there exists a Palais-Smale sequence
$\{u_n\}\subset S_c$ for $J|_{S_c}$ at the level $m^\ast$
with the properties $|u_n^-|_p\to 0$ and
\begin{equation*}P(u_n)\to 0,\end{equation*}
as $n\to\infty$, where $u^-=\min\{u,0\}$.
\end{lemma}

\begin{proof}[\bf Proof]
We introduce a stretched functional $\tilde{J}: E\times\R\to\R$
in the spirit of \cite{lj} as
$$\tilde{J}(u,t)=J(u^t)
=\frac{e^{pt}}{p}|\nabla u|_p^p
-\frac{e^{p^\ast t}}{p^\ast}|u|_{p^\ast}^{p^\ast}
-\mu\frac{e^{\gamma_q q t}}{q}|u|_q^q.$$
Let us consider the set
$$\widetilde{\Gamma}
=\left\{\tilde{\gamma}\in C([0,1],S_c\times\R):
\tilde{\gamma}(0)\in \cM^+\times\{0\},\
\tilde{\gamma}(1)\in J^{2m(c)}\times\{0\}\right\}.$$
It is clear that if
$\gamma\in\Gamma$ then $\tilde{\gamma}:=(\gamma,0)\in\widetilde{\Gamma}$
and $\tilde{J}(\tilde{\gamma}(t))=J(\gamma(t))$ for $t\in[0,1]$;
while if $\tilde{\gamma}=(\tilde{\gamma}_1,\tilde{\gamma}_2)\in\widetilde{\Gamma}$
then
from Lemma \ref{lem2.3},
$\gamma(\cdot):=\tilde{\gamma}_1(\cdot)^{\tilde{\gamma}_2(\cdot)}\in\Gamma$
and $J(\gamma(t))=\tilde{J}(\tilde{\gamma}(t))$ for $t\in[0,1]$.
Therefore, we have
\begin{equation}\label{eq4.34}m^\ast=\inf_{\tilde{\gamma}\in\widetilde{\Gamma}}\sup_{t\in[0,1]}
\tilde{J}(\tilde{\gamma}(t)).\end{equation}

Under the definition of \eqref{eq4.34},
we observe that for $\va_n=\frac{1}{n^2}$ there exists $\gamma_n\in\Gamma$ such that
$$\sup_{t\in[0,1]}J(\gamma_n(t))\leq m^\ast+\frac{1}{n^2}.$$
Replacing $\gamma_n$ by $|\gamma_n|$ if necessary, we may assume
that $\gamma_n(t)\geq0$ in $\R^N$ for all $t\in[0,1]$.
In addition, setting $\tilde{\gamma}_n=(\gamma_n,0)\in\widetilde{\Gamma}$,
we have
$$\sup_{t\in[0,1]}\tilde{J}(\tilde{\gamma}_n(t))\leq m^\ast+\frac{1}{n^2}.$$

For any $\tilde{\gamma}=(\tilde{\gamma}_1,\tilde{\gamma}_2)\in\widetilde{\Gamma}$, let us consider the
function $P_{\tilde{\gamma}}:[0,1]\to\R$ given by
$$P_{\tilde{\gamma}}(\tau)=P(\tilde{\gamma}_1(\tau)^{\tilde{\gamma}_2(\tau)}).$$
Since ${\tilde{\gamma}}(0)=(\tilde{\gamma}_1(0),\tilde{\gamma}_2(0))\in (\cM^+,0)$,
there holds $t_{\tilde{\gamma}_1(0)}>s_{\tilde{\gamma}_1(0)}=0$ according to Lemma \ref{lem4.3}.
Observing that $J(\tilde{\gamma}_1(1))=\tilde{J}({\tilde{\gamma}}(1))<2m(c)$ we have
$t_{\tilde{\gamma}_1(1)}<0$ by virtue of Lemma \ref{lem4.12}.
Moreover, Lemma \ref{lem2.3} yields the map $\tau\mapsto\tilde{\gamma}_1(\tau)^{\tilde{\gamma}_2(\tau)}$ is continuous from $[0,1]\to E$.
Hence, there exists $\tau_{{\tilde{\gamma}}}\in(0,1)$ such that $P_{\tilde{\gamma}}(\tau_{{\tilde{\gamma}}})=0$,
which implies
\begin{equation}\label{eq4.35}\max_{\tilde{\gamma}([0,1])}\tilde{J}\geq \tilde{J}({\tilde{\gamma}}(\tau_{{\tilde{\gamma}}}))\geq\inf_{\cM^-}J(u)=m^\ast.\end{equation}
It follows from Lemma \ref{lem4.3} that
\begin{equation}\label{eq4.36}m^\ast>0\geq \sup_{\cM^+\cup J^{2m(c)}}J(u)=\sup_{(\cM^+,0)\cup (J^{2m(c)},0)}\tilde{J}(u,t).\end{equation}
In the following, we will apply Lemma \ref{lem2.4} to achieve our result. For this purpose, let
$$X=S_c\times\R,\ {\mathcal F}=\{\tilde{\gamma}([0,1]):\tilde{\gamma}\in\widetilde{\Gamma}\},\
B=(\cM^+,0)\cup (J^{2m(c)},0),$$
$$F=\{(u,s)\in S_c\times\R|\tilde{J}(u,s)\geq m^\ast\},\ A=\tilde{\gamma}([0,1]),\
A_n=\tilde{\gamma}_n([0,1])=\gamma_n([0,1])\times\{0\}.$$
We need to checked that ${\mathcal F}$ is a homotopy stable family of compact subsets of $X$ with extended closed boundary $B$ and $F$ satisfies the assumptions (1) and (2) in Lemma \ref{lem2.4}.
In fact, for every $\tilde{\gamma}\in\widetilde{\Gamma}$, since $\tilde{\gamma}(0)\in(\cM^+,0)$ and
$\tilde{\gamma}(1)\in(J^{2m(c)},0)$, we have $\tilde{\gamma}(0), \tilde{\gamma}(1)\in B$.
For any set $A$ in
${\mathcal F}$ and any $\eta\in C([0,1]\times X; X)$ satisfying $\eta(t,x)=x$ for all $(t,x)\in(\{0\}\times X)\cup
([0,1]\times B)$, there holds that $\eta(1,\tilde{\gamma}(0))=\tilde{\gamma}(0)$, $\eta(1,\tilde{\gamma}(1))=\tilde{\gamma}(1)$. Hence, we have that $\eta(\{1\}\times A)\in{\mathcal F}$.
Now, in view of \eqref{eq4.35}, we have $A\cap F\neq\emptyset$. Meanwhile, $F\cap B=\emptyset$ by virtue of \eqref{eq4.36}. Hence, we can deduce that the assumptions (1) and (2) in Lemma \ref{lem2.4} are valid.

Therefore,
using Lemma \ref{lem2.4}
yields the existence of a sequence $\{(v_n,t_n)\}\subset S_c\times\R$
such that, as $n\to+\infty$,
\begin{equation}\label{eq4.37}
|t_n|+\text{dist}(v_n,{\gamma}_n([0,1]))\to 0,
\end{equation}
\begin{equation}\label{eq4.38}
\tilde{J}(v_n,t_n)\to m^\ast,
\end{equation}
\begin{equation}\label{eq4.39}
(\tilde{J}|_{S_c\times\R})'(v_n,t_n)\to 0.
\end{equation}
Moreover, we can infer that $\tilde{J}(v_n,t_n)=\tilde{J}(v_n^{t_n},0)$ and
\begin{equation}\label{eq4.40}
\langle(\tilde{J}|_{S_c\times\R})'(v_n,t_n),(\varphi,s)\rangle
=\langle(\tilde{J}|_{S_c\times\R})'(v_n^{t_n},0), (\varphi^{t_n},s)\rangle
\end{equation}
for $(\varphi,s)\in E\times\R$ with $\int_{\R^N}v_n\varphi=0$.
Recording that $u_n=v_n^{t_n}\in S_c$, we see from \eqref{eq4.38} that
$$J(u_n)=\tilde{J}(v_n^{t_n},0)=\tilde{J}(v_n,t_n)\to m^\ast,\ \
\text{as}\ n\to\infty.$$
Taking $({\bf 0},1)$ as a test function in \eqref{eq4.40}, we deduce
from \eqref{eq4.39} that
$$P(u_n)=\partial_t\tilde{J}(u_n,0)\to 0,\ \ \text{as}\ n\to\infty.$$
For $w\in E$ with $\int_{\R^N}v_n^{t_n}w=0$, we immediately deduce from \eqref{eq4.37},
\eqref{eq4.39} and Lemma \ref{lem2.2} that $(J|_{S_c})'(u_n)\to 0$ as $n\to\infty$ provided taking $(w^{-t_n},0)$
as a test function in \eqref{eq4.40}.
In view of \eqref{eq4.37} again,
$|u_n^-|_p^p=|(v_n^-)^{t_n}|_p^p=|v_n^-|_p^p\to 0$ as $n\to\infty$, as desired.
The proof is completed.
\end{proof}

\begin{proof}[\bf Proof of Theorem \ref{the1.1} (Part III)]
By Lemma \ref{lem4.15}, there exists a sequence $\{u_n\}\subset S_c$ with
the following properties
$$J(u_n)\to m^*,\ \ (J|_{S_c})'(u_n)\to 0,\ \ |u_n^-|_p\to 0,\ \
P(u_n)\to 0,\ \ \text{as}\ n\to\infty.$$
If alternative $(i)$ in Lemma \ref{lem3.3} occurs, then
$u_n\rightharpoonup u$ in $E$ for some $u\neq 0$ and
\begin{align*}
m^\ast-\Lambda
&\geq J(u).
\end{align*}
In the case $J(u)\geq0$, we have
$$m^\ast-\Lambda\geq 0>m(c),$$
yielding a contradiction with Lemma \ref{lem4.14}; while in
the case $J(u)<0$, we have $|\nabla u|_p<t^*$ and then, by
$|u|_p\leq c$ and Remark \ref{rem4.1} that
$$J(u)\geq m(|u|_p)\geq m(c),$$
which is also contradicts the result of Lemma \ref{lem4.14}.
Therefore, alternative $(ii)$
in Lemma \ref{lem3.3} holds true and henceforth $u_n\to u$ in
$E$ with $u$ being a ground state solution of \eqref{eq1.1} and \eqref{eq1.3} for some
$\lambda<0$. Moreover, by $|u_n^-|_p\to 0$ as
$n\to\infty$ and then similarly as in proof of Theorem \ref{the1.1}(Part I), we conclude that $u$ is positive.
The proof is completed.
\end{proof}

\section{$L^p$-critical and supercritical perturbations}

Throughout this section, we always assume
that
$$p+p^2/N\leq q<p^\ast,\ \ c,\mu>0,\ \mu c^{(1-\gamma_q)q}<\alpha(q).$$
We wish to investigate the
mountain pass geometry of $J$ on $S_c$. For this reason, we recall that
$$D_k:=D_{k}(c)=\{u\in S_c : |\nabla u|_p<k\},\ \ \text{for}\ k>0.$$

\begin{lemma}\label{lem5.1}
Suppose that $p+p^2/N\leq q<p^\ast$ and $\mu c^{(1-\gamma_q)q}<\alpha(q)$
when $q=p+p^2/N$.\\
$(1)$ There exist two positive numbers $k_1<k_2$ such that
$$0<\sup_{u\in \overline{D_{k_1}}}J(u)<\inf_{u\in \partial D_{k_2}}J(u)$$
and
$$J(u)>0\ \ \text{and}\ \ P(u)>0,\ \ \text{for}\ u\in D_{k_2}.$$
$(2)$ There exists $u_0\in S_c\setminus A_{k_2}$ such that $J(u_0)<0$.
\end{lemma}

\begin{proof}[\bf Proof]
We focus our attention
on the following two aspects.

$(1)$ It follows from Lemma \ref{lem1.1} and the Sobolev inequality
that, for $u\in S_c$,
\begin{align*}
J(u)
&=\frac{1}{p} |\nabla u|_p^p
-\frac{1}{p^\ast}|u|_{p^\ast}^{p^\ast}-\frac{\mu}{q}|u|_q^q\\
&\geq\frac{1}{p} |\nabla u|_p^p
-\frac{S^{-\frac{{p^\ast}}{p}}}{{p^\ast}}|\nabla u|_p^{p^\ast}
-\frac{\mu}{q}C(q)c^{(1-\gamma_q)q}|\nabla u|_p^{\gamma_q q}
\end{align*}
and
\begin{align*}
P(u)
&=|\nabla u|_p^p-|u|_{p^\ast}^{p^\ast}-\mu \gamma _q|u|_q^q\\
&\geq |\nabla u|_p^p-S^{-\frac{p^\ast}{p}}|\nabla u|_p^{p^\ast}
-\mu \gamma _qC(q)c^{(1-\gamma_q)q}|\nabla u|_p^{\gamma_q q}.
\end{align*}
In particular, it is also clear that
$$J(u)\leq\frac{1}{p} |\nabla u|_p^p.$$
Observe that $p\leq \gamma_q q<p^\ast$ and
$$\frac{1}{p}>\frac{\mu}{q}C(q)c^{(1-\gamma_q)q}\ \
\text{when}\ \gamma_q q=p.$$
Thereby, taking two small positive numbers $k_1<k_2$, we arrive at the
desired result.

$(2)$ It is not difficult
to check that, for $u\in S_c$,
$$\lim_{t\to +\infty}|\nabla u^t|_p=+\infty\ \ \text{and}\ \
\lim_{t\to +\infty}J(u^t)=-\infty.$$
If, in addition, we choose $u_0=u^t$ with $t>0$ sufficiently large,
then the desired result is obtained and this completes the proof.
\end{proof}

By Lemma \ref{lem5.1}, we can define the mountain pass level of
the functional $J$ on $S_c$ by
\begin{equation}\label{eq5.1}
m^\ast:=m^\ast( c,\mu)
=\inf_{\gamma\in\Gamma}\sup_{t\in[0,1]}J(\gamma(t)),
\end{equation}
where
\begin{equation*}
\Gamma:=\Gamma(c,\mu)=\left\{\gamma\in C([0,1],S_c):
\gamma(0)\in \overline{D_{k_1}},\ J(\gamma(1))\leq0\right\}.
\end{equation*}
Clearly, $m^\ast\geq\inf_{u\in \partial D_{k_2}}J(u)>0$.

\subsection{The case $q=p+p^2/N$}
Let us keep in mind that $q=\bar{q}=p+p^2/N$ throughout this
subsection.

\begin{lemma}\label{lem5.2}
Let $a_2,\,b_2>0$. Then the function
$$\xi(t)=a_2e^{pt}-b_2e^{p^\ast t},\ \ t\in\R$$
has a unique critical point at which $\xi$ achieves its maximum.
\end{lemma}

We observe that $\gamma_{\bar{q}}\bar{q}=p$ and
$\mu c^{(1-\gamma_{\bar{q}}){\bar{q}}}<\alpha(\bar{q})$, hence there holds
\begin{equation}\label{eq5.2}
\gamma _{\bar{q}}C({\bar{q}})\mu c^{(1-\gamma_{\bar{q}}){\bar{q}}}<1.
\end{equation}

\begin{lemma}\label{lem5.3}
It results
$\cM^0=\emptyset$.
\end{lemma}

\begin{proof}[\bf Proof]
Suppose by contradiction that $\cM^0\neq\emptyset$. Then there exists
$u\in\cM^0$ satisfying
$$|\nabla u|_p^p-|u|_{p^\ast}^{p^\ast}
-\mu\gamma_{\bar{q}}|u|_{\bar{q}}^{\bar{q}}=0$$
and
$$p|\nabla u|_p^p-{p^\ast}|u|_{p^\ast}^{p^\ast}
-p\mu \gamma_{\bar{q}}|u|_{\bar{q}}^{\bar{q}}=0.$$
Thus, gathering Lemma \ref{lem1.1} and \eqref{eq5.2} we see that
$$|\nabla u|_p^p
=\mu \gamma_{\bar{q}}|u|_{\bar{q}}^{\bar{q}}
\leq \gamma_{\bar{q}} C({\bar{q}})\mu c^{(1-\gamma_{\bar{q}}){\bar{q}}}
|\nabla u|_p^{p}<|\nabla u|_p^p,$$
which is absurd and therefore the thesis follows. The proof is completed.
\end{proof}

\begin{lemma}\label{lem5.4}
For every $u\in S_c$, there exists a unique $t_u\in\R$ such
that $u^{t_u}\in \cM$ and $J(u^{t_u})=\max_{t\in\R}J(u^t)$.
Moreover, we have\\
$(i)$ $\cM=\cM^-$;\\
$(ii)$ $P(u)<0$ if and only if $t_u<0$;\\
$(iii)$ the map $u\mapsto t_u$ is of class $C^1$.
\end{lemma}

\begin{proof}[\bf Proof]
Recall that $\gamma_{\bar{q}}\bar{q}=2$.
It is immediate to check that, for $u\in S_c$,
$$\Psi_u(t)=J(u^t)
=\left(\frac{1}{p}|\nabla u|_p^p
-\frac{\mu}{\bar{q}}|u|_{\bar{q}}^{\bar{q}}\right)e^{pt}
-\frac{|u|_{p^\ast}^{p^\ast}}{p^\ast}e^{p^\ast t}$$
and
$$P(u^t)=\left(|\nabla u|_p^p
-\mu\gamma_{\bar{q}}|u|_{\bar{q}}^{\bar{q}}\right)e^{pt}
-|u|_{p^\ast}^{p^\ast}e^{{p^\ast} t}.$$
It is well known that $\Psi_u'(t)=0\Leftrightarrow P(u^t)=0\Leftrightarrow u^t\in\cM$.
Thus, exploiting again the fact that
$$\frac{1}{p}|\nabla u|_p^p-\frac{\mu}{\bar{q}}|u|_{\bar{q}}^{\bar{q}}
\geq\frac1p\left(1-\gamma_{\bar{q}}C(\bar{q})
\mu c^{(1-\gamma_{\bar{q}})\bar{q}}\right)|\nabla u|_p^p>0.$$
We deduce, from Lemma \ref{lem5.2}, that there exists a unique
$t_u\in\R$ such that $\Psi_u'(t_u)=0$ and
$J(u^{t_u})=\max_{t\in\R}J(u^t)$. Thereby,
$u^{t_u}\in \cM$.

If $u\in \cM$, then $t_u=0$ is the maximum point of $\Psi_u$.
According to Lemma \ref{lem5.3}, we have $\Psi_u''(0)<0$
and then $u\in \cM^-$. Hence $\cM=\cM^-$. Observing that
$\Psi_u'(t)<0$ if and only if $t>t_u$, we have $P(u)=\Psi_u'(0)<0$
if and only if $t_u<0$. Now, let us define the functional $\Phi: E\times\R\to\R$ by
$\Phi(u,t)=\Psi_u'(t)$, it follows that $\Phi$ is of class
$C^1$, $\Phi(u, t_u)=0$ and $\partial_t\Phi(u,t_u)=\Psi''_u(t_u)<0$.
Making use of the Implicit Function Theorem, we find that the map
$u\mapsto t_u$ is of class $C^1$.
\end{proof}

\begin{lemma}\label{lem5.5}We have
$m:=m(c,\mu)=\inf_{u\in\cM}J(u)>0$.
\end{lemma}

\begin{proof}[\bf Proof]
Lemma \ref{lem5.4} implies that $\cM\neq\emptyset$.
We note that, for $u\in\cM$,
$$|\nabla u|_p^p
=|u|_{p^\ast}^{p^\ast}+\mu \gamma_{\bar{q}}|u|_{\bar{q}}^{\bar{q}}
\leq S^{-\frac{p^\ast}{p}}|\nabla u|_p^{p^\ast}
+\gamma _{\bar{q}}C({\bar{q}})\mu
c^{(1-\gamma_{\bar{q}}){\bar{q}}}|\nabla u|_p^p,$$
from which and \eqref{eq5.2} it results that
\begin{equation}\label{eq5.3}
\inf_{u\in\cM}|\nabla u|_p>0.
\end{equation}
Obviously, it is not difficult to check, for $u\in\cM$,
\begin{align*}J(u)&=J(u)-\frac{1}{p^\ast}P(u)
=\frac{1}{N}|\nabla u|_p^p
-\frac{\mu\gamma _{\bar{q}}}{N}|u|_{\bar{q}}^{\bar{q}}\\
&\geq\frac1N\left(1-\gamma _{\bar{q}}C({\bar{q}})
\mu c^{(1-\gamma_{\bar{q}}){\bar{q}}}\right)|\nabla u|_p^p,\end{align*}
which whence combined with \eqref{eq5.3} concludes the proof.
\end{proof}

\begin{lemma}\label{lem5.6}
There holds $m=m^\ast$.
\end{lemma}

\begin{proof}[\bf Proof]
For any $u\in \cM$,  it is not difficult to check that
$$\lim_{t\to-\infty}|\nabla u^t|_2^2=0,\ \
\lim_{t\to +\infty}|\nabla u^t|_2^2=+\infty,\ \
\lim_{t\to +\infty}J(u^t)=-\infty.$$
Clearly, there exist $t_1=t_1(u)<0$ and $t_2=t_2(u)>0$ such
that $u^{t_1}\in D_{k_1}$, $u^{t_2}\in S_c\backslash D_{k_2}$
and $J(u^{t_2})<0$. We define for any $t\in[0,1]$ by $\gamma(t)=u^{(1-t)t_1+tt_2}$.
Clearly, in view of Lemma \ref{lem2.3}, $\gamma\in\Gamma$ and then
$\sup_{t\in[0,1]}J(\gamma(t))=J(u)$, from which it ensures
that $m\geq m^\ast$.

From now on we focus on $m\leq m^\ast$, it suffices to verify that
$\gamma([0,1])\cap\cM\neq\emptyset$ for any $\gamma\in\Gamma$.
Since
$$J(u)-\frac{1}{p^\ast}P(u)
\geq\frac1N\left(1-\gamma _{\bar{q}}C({\bar{q}})
\mu c^{(1-\gamma_{\bar{q}}){\bar{q}}}\right)|\nabla u|_p^p>0,
\ \ \text{for}\ u\in S_c,$$
it follows that $P(\gamma(1))< p^\ast J(\gamma(1))\leq0$ for any $\gamma\in\Gamma$.
Note in particular that $P(\gamma(0))>0$.
Furthermore, if we define
$t_1=\inf\{t\in[0,1): P(\gamma(s))<0\ \text{for}\ s\in(t,1]\}$,
then there must be $P(\gamma(t_1))=0$, which whence implies that
$\gamma(t_1)\in\gamma([0,1])\cap\cM$. The proof is complete.
\end{proof}

\begin{lemma}\label{lem5.7} Suppose that $p<N^{2/3}$. Then
$0<m<\frac{1}{N}S^{N/p}$.
\end{lemma}

\begin{proof}[\bf Proof]
The proof is similar to the following Lemma \ref{lem5.13}.
\end{proof}
\begin{lemma}\label{lem5.8}
Suppose that $q=p+p^2/N$ and $\mu c^{(1-\gamma_q)q}<\alpha(q)$. Then there exists a Palais-Smale sequence
$\{u_n\}\subset S_c$ for $J|_{S_c}$ at the level $m^\ast$
with the following properties
$$|u_n^-|_p\to 0\ \ \text{and}\ \ P(u_n)\to 0,\ \ \text{as}\ n\to\infty.$$
\end{lemma}
\begin{proof}[\bf Proof]
The proof is similar to the proof of Lemma \ref{lem4.15}.
\end{proof}

\begin{proof}[\bf Proof of Theorem \ref{the1.2}]
According to Lemmas \ref{lem5.6} and \ref{lem5.8}, there
exists a sequence $\{u_n\}\subset S_c$ with the following
properties
$$J(u_n)\to m,\ \ (J|_{S_c})'(u_n)\to 0,\ \ |u_n^-|_p\to 0,\ \
P(u_n)\to 0,\ \ \text{as}\ n\to\infty.$$
If alternative $(i)$ in Lemma \ref{lem3.3} occurs, then
$u_n\rightharpoonup u$ in $E$ for some $u\neq0$ and, by applying
\eqref{eq3.5} and \eqref{eq5.2},
\begin{align*}
m-\frac{1}{N}S^{N/p}&\geq J(u)=\frac{1}{N}|\nabla u|_p^p
-\frac{\mu\gamma_{\bar{q}}}{N}|u|_{\bar{q}}^{\bar{q}}\\
&\geq\frac1N\left(1-\gamma_{\bar{q}}C({\bar{q}})
\mu c^{(1-\gamma_{\bar{q}}){\bar{q}}}\right)|\nabla u|_p^p
>0,
\end{align*}
yielding a contradiction with Lemma \ref{lem5.7}. Therefore,
alternative $(ii)$ in Lemma \ref{lem3.3} holds true and
hence $u_n\to u$ in $E$ with $u$ being a ground state
solution of \eqref{eq1.1} and \eqref{eq1.3} for some $\lambda<0$.
Moreover, we conclude by observing that $|u_n^-|_p\to 0$ as
$n\to\infty$ and arguing as the proof of Theorem \ref{the1.1}(Part I), that $u$ is positive. The proof
is complete.
\end{proof}

\subsection{The case $p+p^2/N<q<p^\ast$}

This subsection is devoted to consider the case
$p+p^2/N<q<p^\ast$,
which indicates that $p<\gamma_q q<p^\ast$. Similar to Lemma
\ref{lem5.2}, we have the following lemma.

\begin{lemma}\label{lem5.9}
Let $a_3,\,b_3>0$, $d_3\in \R$ and $p<\tau<p^\ast$. Then
the function
$$\tilde{\xi}(t)=a_3e^{pt}-b_3e^{p^\ast t}-d_3 e^{\tau t},
\ \ t\in\R$$
has a unique critical point at which $\tilde{\xi}$ achieves its
maximum.
\end{lemma}

\begin{lemma}\label{lem5.10}
For every $u\in S_c$, there exists a unique $t_u\in\R$ such
that $u^{t_u}\in \cM$ and
$J(u^{t_u})=\max_{t\in\R}J(u^t)$. Moreover, we have\\
$(i)$ $\cM=\cM^-$;\\
$(ii)$ $P(u)<0$ if and only if $t_u<0$;\\
$(iii)$ the map $u\mapsto t_u$ is of class $C^1$.
\end{lemma}

\begin{proof}[\bf Proof]
For $u\in S_c$, we have
$$\Psi_u(t)=J(u^t)=\frac{e^{pt}}{p}|\nabla u|_p^p
-\frac{e^{p^\ast t}}{p^\ast}|u|_{p^\ast}^{p^\ast}
-\mu\frac{e^{\gamma_q q t}}{q}|u|_q^q$$
and
$$P(u^t)=e^{pt}|\nabla u|_p^p
-e^{{p^\ast} t}|u|_{p^\ast}^{p^\ast}-\mu \gamma_qe^{\gamma_q q t}|u|_q^q.$$
It is obvious that
$\Psi_u'(t)=0\Leftrightarrow P(u^t)=0\Leftrightarrow u^t\in\cM$.
Standard computations together with Lemma \ref{lem5.9} shows that there exists a unique $t_u\in\R$
such that $\Psi_u'(t_u)=0$ and
$J(u^{t_u})=\max_{t\in\R}J(u^t)$. Hence, $u^{t_u}\in \cM$.

We remark that, letting $u\in\cM$, a straightforward computation shows
$$\Psi''_u(0)=\Psi''_u(0)-pP(u)
=-(p^\ast-p)|u|_{p^\ast}^{p^\ast}+(p-\gamma_q q)\mu \gamma_q|u|_q^q
<0,$$
whence it follows that $u\in\cM^-$ and henceforth $\cM=\cM^{-}$. Observing
that $\Psi_u'(t)<0$ if and only if $t>t_u$, we have
$P(u)=\Psi_u'(0)<0$ if and only if $t_u<0$. If we define
$\Phi: E\times\R\to\R$ by $\Phi(u,t)=\Psi_u'(t)$, then
$\Phi$ is of class $C^1$, $\Phi(u, t_u)=0$ and
$\partial_t\Phi(u,t_u)=\Psi''_u(t_u)<0$. Now, let us remark that, by applying the Implicit
Function Theorem, the map $u\mapsto t_u$ is of class
$C^1$.
\end{proof}

\begin{lemma}\label{lem5.11}
$m:=m(c,\mu)=\inf_{u\in\cM}J(u)>0$.
\end{lemma}

\begin{proof}[\bf Proof]
Lemma \ref{lem5.10} indicates that $\cM\neq\emptyset$.
Notice in particular that, for $u\in\cM$,
\begin{align*}
0&=|\nabla u|_p^p-|u|_{p^\ast}^{p^\ast}-\mu \gamma _q|u|_q^q\\
&\geq |\nabla u|_p^p-S^{-\frac{p^\ast}{p}}|\nabla u|_p^{p^\ast}
-\mu\gamma_qC(q)c^{(1-\gamma_q)q}|\nabla u|_p^{\gamma_q q},
\end{align*}
whence it follows that
\begin{equation}\label{eq5.4}
\inf_{u\in\cM}|\nabla u|_p>0.
\end{equation}
Therefore, it is immediate to check that, for $u\in\cM$,
\begin{align*}J(u)&=J(u)-\frac{1}{\gamma_qq}P(u)
=\frac{\gamma_qq-p}{p\gamma_qq}|\nabla u|_p^p+\frac{{p^\ast}-\gamma_qq}{{p^\ast}\gamma_qq}|u|_{p^\ast}^{p^\ast}\\
&\geq\frac{\gamma_qq-p}{p\gamma_qq}|\nabla u|_p^p.\end{align*}
This, together with \eqref{eq5.4}, implies the desired conclusion.
\end{proof}

\begin{lemma}\label{lem5.12}
$m=m^\ast$, where $m^\ast$ is defined in \eqref{eq5.1}.
\end{lemma}

\begin{proof}[\bf Proof]
The statement follows straightforwardly from the fact that
\begin{align*}J(u)-\frac{1}{p}P(u)=\frac{1}{N}|u|_{p^\ast}^{p^\ast}
+\frac{\mu}{q}\left(\frac{\gamma_qq}{p}-1\right)|u|_q^q> 0,\ \ \text{for}\ u\in S_c.\end{align*}
Arguing as in the proof of Lemma \ref{lem5.6}, we immediately
achieve the result.
\end{proof}

\begin{lemma}\label{lem5.13}
Assume that $p\leq N^{2/3}$. Then $0<m<\frac{1}{N}S^{N/p}$.
\end{lemma}

\begin{proof}[\bf Proof]
{\bf Case 1.} $p\leq N^{1/2}$. 
Let $v_\va=(\tau)^{\frac{N-p}{p}}u_\va(\tau x)$, where $\tau=c^{-1}|u_\va|_p$.
Then
$$|v_\va|_p=c,\ |\nabla v_\va|_2=|\nabla u_\va|_2,\ |v_\va|_{p^\ast}=|u_\va|_{p^\ast}$$
and for $p<q<p^\ast$.
Lemmas \ref{lem5.10}
and \ref{lem5.11} imply there exists a unique $t_\va\in\R$ such
that $v_\va^{t_\va} \in\cM$ and $0<m\leq J(v_\va^{t_\va})$.
To conclude the proof, it suffices to show that
$J(v_\va^{t_\va})<\frac{1}{N}S^{N/p}$ for $\va>0$ sufficiently small.

According to $p\leq N^{1/2}$, there holds $p^\ast(1-1/p)\leq p<q$.
As a consequence of \eqref{eq2.2}, \eqref{eq2.3}, \eqref{eq2.4} and \eqref{eq5.4},
there exists a constant $C>0$ such that, as $\va\to 0^+$,
\begin{align*}
e^{\gamma_q q t_{\va}}|v_\va|_q^q
&=e^{\gamma_q q t_{\va}}\cdot
\tau^{-(1-\gamma_q)q}|u_\va|_q^q\\
&=\frac{|\nabla v_\va^{t_\va}|_p^{\gamma_q q}}{|\nabla u_\va|_p^{\gamma_q q}}
\cdot c^{(1-\gamma_q)q}\frac{|u_\va|_q^q}{|u_\va|_p^{(1-\gamma_q)q}}
\geq \left\{\begin{array}{ll}
C,\ & p<N^{1/2},\\
C|\ln \va|^{\frac{-(1-\gamma_q)q}{p}},&p=N^{1/2},
\end{array}\right.
\end{align*}
which jointly with Lemma \ref{lem2.1} leads to
\begin{align*}J(v_\va^{t_\va})
&=\frac{a_\va^pe^{pt_\va}}{p}|\nabla u_\va|_p^p
-\frac{a_\va^{p^\ast}e^{{p^\ast} t_\va}}{p^\ast}|u_\va|_{p^\ast}^{p^\ast}
-\mu\frac{e^{\gamma_q q t_\va}}{q}|v_\va|_q^q\\
&\leq \frac{1}{N}S^{N/p}+O(\va^{\frac{N-p}{p}})-\frac{\mu }{q}\left\{\begin{array}{ll}
C,\ & p<N^{1/2},\\
C|\ln \va|^{\frac{-(1-\gamma_q)q}{p}},&p=N^{1/2}.
\end{array}\right.\end{align*}
As a consequence, letting $\va>0$ small yields the desired conclusion.

{\bf Case 2.} $N^{1/2}<p<N^{2/3}$.
Let $v_\va$ be as in \eqref{eq2.1}. Lemmas \ref{lem5.10}
and \ref{lem5.11} imply there exists a unique $t_\va\in\R$ such
that $v_\va^{t_\va} \in\cM$ and $0<m\leq J(v_\va^{t_\va})$.
To conclude the proof, it suffices to show that
$J(v_\va^{t_\va})<\frac{1}{N}S^{N/p}$ for $\va>0$ sufficiently small.

According to $N^{1/2}<p<N^{2/3}$, there holds $p<p^\ast(1-1/p)<p+p^2/N$.
As a consequence of \eqref{eq2.2}, \eqref{eq2.3}, \eqref{eq2.4} and \eqref{eq5.4},
there exists a constant $C>0$ such that, as $\va\to 0^+$,
\begin{align*}
e^{\gamma_q q t_{\va}}|v_\va|_q^q
&=a_\va^{\gamma_q q} e^{\gamma_q q t_{\va}}\cdot
a_\va^{(1-\gamma_q)q}|u_\va|_q^q\\
&=\frac{|\nabla v_\va^{t_\va}|_p^{\gamma_q q}}{|\nabla u_\va|_p^{\gamma_q q}}
\cdot c^{(1-\gamma_q)q}\frac{|u_\va|_q^q}{|u_\va|_p^{(1-\gamma_q)q}}
\geq C \va^{\frac{(pN+pq-qN)(p^2-N)}{p^3}},
\end{align*}
which jointly with Lemma \ref{lem2.1} leads to
\begin{align*}J(v_\va^{t_\va})
&=\frac{a_\va^pe^{pt_\va}}{p}|\nabla u_\va|_p^p
-\frac{a_\va^{p^\ast}e^{{p^\ast} t_\va}}{p^\ast}|u_\va|_{p^\ast}^{p^\ast}
-\mu\frac{e^{\gamma_q q t_\va}}{q}|v_\va|_q^q\\
&\leq \frac{1}{N}S^{N/p}+O(\va^{\frac{N-p}{p}})-\frac{\mu }{q}C\va^{\frac{(pN+pq-qN)(p^2-N)}{p^3}}.\end{align*}
The condition $N^{1/2}<p<N^{2/3}$ implies that $\frac{(pN+pq-qN)(p^2-N)}{p^3}<\frac{N-p}{p}$.
As a consequence, letting $\va>0$ small yields the desired conclusion.
\end{proof}
\begin{lemma}\label{lem5.14}
Suppose that $p+p^2/N< q<p^\ast$. Then there exists a Palais-Smale sequence
$\{u_n\}\subset S_c$ for $J|_{S_c}$ at the level $m^\ast$
with the following properties
$$|u_n^-|_p\to 0\ \ \text{and}\ \ P(u_n)\to 0,\ \ \text{as}\ n\to\infty.$$
\end{lemma}
\begin{proof}[\bf Proof]
The proof is similar to the proof of Lemma \ref{lem4.15}.
\end{proof}

\begin{proof}[\bf Proof of Theorem \ref{the1.3}]
We can apply Lemmas \ref{lem5.12} and \ref{lem5.14} to conclude that there is a sequence
$\{u_n\}\subset S_c$ with the following properties
$$J(u_n)\to m,\ \ (J|_{S_c})'(u_n)\to 0,\ \ |u_n^-|_p\to 0,
\ \ P(u_n)\to 0,\ \ \text{as}\ n\to\infty.$$
According to Lemma \ref{lem5.13}, one of the two alternatives
in Lemma \ref{lem3.3} occurs. Suppose that alternative
$(i)$ occurs, it is elementary to realize that $u_n\rightharpoonup u$ in $E$ for some
$u\neq 0$ and, by virtue of \eqref{eq3.5},
\begin{align*}
m-\frac{1}{N}S^{N/p}
&\geq J(u)
=J(u)-\frac{1}{p}P(u)\\
&=\frac{1}{N}|u|_{p^\ast}^{p^\ast}
+\frac{\mu}{q}\left(\frac{\gamma_qq}{p}-1\right)|u|_q^q>0,
\end{align*}
which contradicts the result of Lemma \ref{lem5.13}. Therefore,
alternative $(ii)$ in Lemma \ref{lem3.3} holds true and
henforth $u_n\to u$ in $E$ with $u$ being a ground state
solution of \eqref{eq1.1} and \eqref{eq1.3} in $E$ for some
$\lambda<0$. Moreover, observing that $|u_n^-|_p\to 0$ as
$n\to\infty$ and
following the same argument of the proof of Theorem \ref{the1.1}(Part I), we deduce that $u$ is positive.
\end{proof}

\section{Asymptotic behavior of normalized solutions}
In this section, we study the asymptotic behavior of
normalized solutions obtained in Theorem \ref{the1.1}
as $\mu\to0^+$. For
simplicity of notations, we always write $m(\mu):=m(c,\mu)$,
$m^\ast(\mu):=m^\ast(c,\mu)$, $\cM_{\mu}:=\cM_{c,\mu}$
and $\cM_{\mu}^\pm:=\cM_{c,\mu}^\pm$ throughout this section.

\subsection{The case $\mu\to 0^+$}
Let $c>0$ and $p<q<p^\ast$ be fixed. We also assume that
$\mu>0$ satisfies \eqref{eq1.6}.

\begin{proof}[\bf Proof of Theorem \ref{the1.4}\,$(i)$]
In the case $p<q<p+p^2/N$, Theorem \ref{the1.1} states that
$u_{\mu}:=u_{c,\mu}$ is the interior local minimizer
of the functional $J_{\mu}$ on $D_{R_0}(c)$, where $R_0=R_0(c,\mu)$
is given by Lemma \ref{lem4.1}. It is easy to verify that
$R_0(c,\mu)\to 0$ as $\mu\to 0^+$ and thus
$|\nabla u_{\mu}|_p\to 0$ as $\mu\to 0^+$. Using Lemma
\ref{lem1.1} and the Sobolev inequality, we derive
\begin{align*}
0>m(\mu)
&=\frac{1}{p} |\nabla u_{\mu}|_p^p
-\frac{1}{p^\ast}|u_{\mu}|_{p^\ast}^{p^\ast}-\frac{\mu}{q}|u_{\mu}|_q^q\\
&\geq\frac{1}{p} |\nabla u_{\mu}|_p^p
-\frac{S^{-\frac{p^\ast}{p}}}{p^\ast}|\nabla u_{\mu}|_p^{p^\ast}
-\frac{\mu}{q}C(q)c^{(1-\gamma_q)q}|\nabla u_{\mu}|_p^{\gamma_q q},
\end{align*}
which implies that $m(\mu)\to 0$ as $\mu\to 0^+$ and the thesis follows.
\end{proof}

Next we consider the case where $\bar{q}=p+p^2/N\leq q<p^\ast$.

\begin{lemma}\label{lem6.1}
There holds
$$\inf_{u\in\cM_{\mu}}J_{\mu}(u)
=\inf_{u\in S_c}\max_{t\in\R}J_{\mu}(u^t).$$
\end{lemma}

\begin{proof}[\bf Proof]
We see from Lemmas \ref{lem5.4} and \ref{lem5.10} that
$$J_{\mu}(u)=\max_{t\in\R}J_{\mu}(u^t)
\geq\inf_{u\in S_c}\max_{t\in\R}J_{\mu}(u^t),\ \
\text{for}\ u\in \cM_{\mu}.$$
On the other hand, if $u\in S_c$ then there is a unique
$t_{u,\mu}:=t_{u,c,\mu}\in\R$ such that
$u^{t_{u,\mu}}\in\cM_{\mu}$ and
$$\max_{t\in\R}J_{\mu}(u^t)=J_{\mu}(u^{t_{u,\mu}})
\geq\inf_{u\in\cM_{\mu}}J_{\mu}(u).$$
The desired conclusion follows easily.
\end{proof}

\begin{lemma}\label{lem6.2}
We have
$$\inf_{u\in\cM_{0}}J_{0}(u)
=\inf_{u\in S_c}\max_{t\in\R}J_{0}(u^t)=\frac{1}{N}S^{N/p}.$$
\end{lemma}

\begin{proof}[\bf Proof]
Similar to Lemma \ref{lem6.1}, we have
$$\inf_{u\in\cM_{0}}J_{0}(u)
=\inf_{u\in S_c}\max_{t\in\R}J_{0}(u^t).$$
Furthermore, Lemma \ref{lem2.1} indicates that
$\inf_{u\in S_c}\max_{t\in\R}J_{0}(u^t)\leq\frac{1}{N}S^{N/p}$. Let
$u\in S_c$ and denote
$$a_1=|\nabla u|_p^p,\ \ b_1=|u|_{p^\ast}^{p^\ast}.$$
It follows from Lemma \ref{lem2.1} and the definition of $S$
that
$$\max_{t\in\R}J_{0}(u^t)
=\frac{a_1^\frac{p^\ast}{p^\ast-p}}{Nb_1^\frac{p}{p^\ast-p}}\geq\frac{1}{N}S^{N/p}.$$
Hence, we have
$\inf_{u\in S_c}\max_{t\in\R}J_{0}(u^t)\geq\frac{1}{N}S^{N/p}$,
concluding the proof.
\end{proof}

\begin{lemma}\label{lem6.3}
Let $\bar{\mu}>0$ be such that
$$\bar{\mu} c^{(1-\gamma_q)q}<\alpha(q).$$
Then the function $\mu\mapsto m(\mu)$ is non-increasing on
$[0,\bar{\mu}]$.
\end{lemma}

\begin{proof}[\bf Proof]
Let $0\leq\mu_1<\mu_2\leq\bar{\mu}$. By Lemma \ref{lem6.1} and \ref{lem6.2}, we have
$$m(\mu_2)
=\inf_{u\in S_c}\max_{t\in\R}J_{\mu_2}(u^t)
\leq\inf_{u\in S_c}\max_{t\in\R}J_{\mu_1}(u^t)
=m(\mu_1).$$
The proof is complete.
\end{proof}

\begin{proof}[\bf Proof of Theorem \ref{the1.4}\,$(ii)$]
Let $c>0$, and let $\bar{\mu}$ satisfy the assumption of Lemma \ref{lem6.3} for this choice of $c$.
In what follows, we will show that the family of positive radial ground states $\{u_{\mu}: 0<\mu<\bar{\mu}\}$
is bounded in $E$.
Note that $P_{\mu}(u_{\mu})=0$. If $q=\bar{q}=p+p^2/N$,
then
\begin{align*}
m(\mu)&=J_{\mu}(u_{\mu})-\frac{1}{p^\ast}P_{\mu}(u_{\mu})\\
&=\frac{1}{N}|\nabla u_{\mu}|_p^p-\frac{p\mu}{N\bar{q}}|u_{\mu}|_{\bar{q}}^{\bar{q}}\\
&\geq\left(\frac{1}{N}-\frac{\gamma_{\bar{q}}C(\bar{q})}{N}
\mu c^{(1-\gamma_{\bar{q}})\bar{q}}\right)|\nabla u_{\mu}|_p^{p};
\end{align*}
while if $p+p^2/N< q<p^\ast$, then
\begin{align*}
m(\mu)&=J_{\mu}(u_{\mu})-\frac{1}{\gamma_qq}P_{\mu}(u_{\mu})\\
&=\Big(\frac{1}{p}-\frac{1}{\gamma_q q}\Big)|\nabla u_{\mu}|_p^p
+\Big(\frac{1}{\gamma_q q}-\frac{1}{p^\ast}\Big)|u_{\mu}|_{p^\ast}^{p^\ast}\\
&\geq\Big(\frac{1}{p}-\frac{1}{\gamma_q q}\Big)|\nabla u_{\mu}|_p^p.
\end{align*}
We see from this and Lemma \ref{lem6.3} that $\{u_{\mu}\}$
is bounded in $E$ with respect to $\mu\in(0,\bar{\mu})$.

Assume up to a subsequence that $u_{\mu}\rightharpoonup u$
in $E$, $u_{\mu}\to u$ in $L^q(\R^N)$, $u_{\mu}\to u$
a.e. in $\R^N$ and $|\nabla u_{\mu}|_p^p\to l\geq 0$ as
$\mu\to 0^+$. Then, by $P_{\mu}(u_{\mu})=0$ again and
the Sobolev inequality, we have
$$l\leq S^{-\frac{N}{N-p}}l^\frac{N}{N-p}.$$
Thus, it is to obtain either $l=0$ or $l\geq S^\frac{N}{p}$.
If $l=0$, then
$u_{\mu}\to 0$ in $D^{1,p}(\R^N)$ and hence
$J_{\mu}(u_{\mu})\to 0$ as $\mu\to 0^+$.
 On the other
hand, Lemma \ref{lem6.3} indicates that
$J_{\mu}(u_{\mu})=m(\mu)\geq m(\bar{\mu})>0$ for
$0<\mu<\bar{\mu}$, which comes to a contradiction. Therefore,
we have $l\geq S^\frac{N}{p}$. Combining
Lemmas \ref{lem6.2} and \ref{lem6.3} yields that
$$\frac{ l}{N}=\lim_{\mu\to 0^+}
\left(J_{\mu}(u_{\mu})-\frac{1}{p^\ast}P_{\mu}(u_{\mu})\right)
=\lim_{\mu\to 0^+}m(\mu)
\leq m(0)=\frac{1}{N}S^{N/p}.$$
Then there must be $l=S^\frac{N}{p}$
and $\lim_{\mu\to 0^+}m(\mu)=\frac{1}{N}S^{N/p}$.

Finally, we claim $u=0$, i.e., $u_{\mu}\rightharpoonup 0$
in $E$ as $\mu\to 0^+$. Indeed, let $\lambda_\mu<0$ be the Lagrange
multiplier associated with the ground state solution $u_{\mu}$.
By $P_{\mu}(u_{\mu})=0$, we have
$$\lambda_\mu c^p=\lambda_\mu|u_{\mu}|_p^p
=|\nabla u_{\mu}|_p^p
-|u_{\mu}|_{p^\ast}^{p^\ast}-\mu|u_{\mu}|_q^q
=\mu(\gamma_q-1)|u_{\mu}|_q^q$$
and then $\lambda_\mu\to 0$ as $\mu\to 0^+$.  Combining this
with $u_{\mu}\rightharpoonup u$ in $E$, a similar argument to the proof of Lemma \ref{lem3.3} shows that $u$
weakly solves
$$-\Delta_p u=u^{p^\ast-1},\ u\geq0,\ {\rm in}\  \R^N,\ u\in E,$$
which
however has only trivial solution in $E$. Therefore, we have
$u=0$.
\end{proof}

\textbf{Statements and Declarations:} There is no competing interests.

\textbf{Acknowledgement:}
The authors would like to express their sincere gratitude to anonymous referees for his/her
constructive comments for improving the quality of this paper.

\end{document}